\documentclass[a4paper,11pt]{amsart}
\usepackage{amssymb, amsmath}
\usepackage{amscd}
\numberwithin{equation}{section}
\usepackage{epsfig}
\usepackage{amsmath}
\usepackage{amsfonts,amssymb,amsopn}

\usepackage{amsthm}
\usepackage{hyperref}
\usepackage{verbatim}
\usepackage{color}
\usepackage[color,all]{xy}

\newcommand{\cL}{{\mathcal L}}
\newcommand{\cM}{{\mathcal M}}
\newcommand{\cN}{{\mathcal N}}
\newcommand{\cO}{{\mathcal O}}
\newcommand{\cS}{{\mathcal S}}
\newcommand{\cU}{{\mathcal U}}
\newcommand{\cW}{{\mathcal W}}
\newcommand{\cQ}{{\mathcal Q}}
\newcommand{\cH}{{\mathcal H}}
\newcommand{\C}{{\mathbb C}}
\newcommand{\hE}{{\widehat{E}}}
\newcommand{\bE}{\overline{E}}
\newcommand{\tM}{\widetilde{M}}
\newcommand{\PP}{{\mathbb P}}
\newcommand{\bZ}{{\mathbb Z}}
\newcommand{\Sym}{\mathrm{Sym}}
\newcommand{\End}{\mathrm{End}}
\newcommand{\Ker}{\mathrm{Ker}}
\newcommand{\Aut}{\mathrm{Aut}\,}
\newcommand{\Image}{\mathrm{Im}\,}
\newcommand{\Iden}{\mathrm{Id}}
\newcommand{\Hom}{\mathrm{Hom}}
\newcommand{\rank}{\mathrm{rk}\,}
\newcommand{\Pic}{\mathrm{Pic}}
\newcommand{\isom}{\xrightarrow{\sim}}
\newcommand{\Kc}{K_{C}}
\newcommand{\Oc}{{{\cO}_{C}}}
\newcommand{\gr}{\mathrm{gr}}
\newcommand{\Quot}{\mathrm{Quot}}
\newcommand{\Quoto}{\mathrm{Quot}^\circ}
\newcommand{\IQ}{\mathrm{IQ}}
\newcommand{\IQo}{\IQ^\circ}
\newcommand{\IQoe}{\IQo_{n,e} (V)}
\newcommand{\IQome}{\IQo_{m,e} (V)}
\newcommand{\st}{{\mathrm{st}}}
\newcommand{\GL}{\mathrm{GL}}
\newcommand{\SL}{\mathrm{SL}}
\newcommand{\SO}{\mathrm{SO}}
\newcommand{\Sp}{\mathrm{Sp}}

\newcommand{\GO}{\mathrm{GO}}
\newcommand{\Spin}{\mathrm{Spin}}
\newcommand{\Or}{\mathrm{O}}
\newcommand{\Gr}{\mathrm{Gr}}
\newcommand{\LG}{\mathrm{LG}}
\newcommand{\LQ}{\mathrm{LQ}}
\newcommand{\OG}{\mathrm{OG}}
\newcommand{\SU}{\cS\cU_C}
\newcommand{\MS}{\cM\cS_C}
\newcommand{\MO}{\cM\cO_C}

\newtheorem{theorem}{{\textbf Theorem}}[section]
\newtheorem{proposition}[theorem]{{\textbf Proposition}}
\newtheorem{corollary}[theorem]{{\textbf Corollary}}
\newtheorem{lemma}[theorem]{{\textbf Lemma}}
\newtheorem{criterion}[theorem]{{\textbf Criterion}}
\newtheorem{conjecture}[theorem]{{\textbf Conjecture}}
\newtheorem{cautionit}[theorem]{{\textbf Caution}}
\newenvironment{caution}{\begin{cautionit}\rm}{\end{cautionit}}
\newtheorem{remit}[theorem]{{\textbf Remark}}

\newenvironment{remark}{\begin{remit}\rm}{\end{remit}}

\title{Low rank orthogonal bundles and quadric fibrations}

\author{Insong Choe}
\address{Department of Mathematics, Konkuk University, 1 Hwayang-dong, Gwangjin-Gu, Seoul 143-701, Korea}
\email{ischoe@konkuk.ac.kr}

\author{George H.\ Hitching}
\address{Oslo Metropolitan University, Postboks 4, St. Olavs plass, 0130 Oslo, Norway}
\email{gehahi@oslomet.no}



\thanks{The first named author  was supported by the National Research Foundation of Korea: NRF-2020R1F1A1A01068699.}

\begin{document}

\begin{abstract} Let $C$ be a curve and $V \to C$ an orthogonal vector bundle of rank $r$. For $r \le 6$, the structure of $V$ can be described using tensor, symmetric and exterior products of bundles of lower rank, essentially due to the existence of exceptional isomorphisms between $\Spin (r , \C)$ and other groups for these $r$. We analyze these structures in detail, and in particular use them to describe moduli spaces of orthogonal bundles. Furthermore, the locus of isotropic vectors in $V$ defines a quadric subfibration $Q_V \subset \PP V$. Using familiar results on quadrics of low dimension, we exhibit isomorphisms between isotropic Quot schemes of $V$ and certain ordinary Quot schemes of line subbundles. In particular, for $r \le 6$ this gives a method for enumerating the isotropic subbundles of maximal  degree of a general $V$, when there are finitely many. \end{abstract}

\maketitle

\section{Introduction}

Let $C$ be a complex projective smooth curve of genus $g \ge 2$, and $L \to C$ a line bundle. A vector bundle $V \to C$ is said to be \textsl{$L$-valued orthogonal} if there is a nondegenerate symmetric bilinear form $\sigma \colon V^{\otimes 2} \to L$. Like symplectic bundles (defined similarly, but with the form being skew-symmetric), orthogonal bundles are interesting examples of decorated vector bundles. Moreover, they are associated bundles of principal $G$-bundles for $G$ an orthogonal group $\SO (r , \C)$, $\Or ( r , \C )$ or $\GO ( r , \C )$, and provide useful tools and intuition for questions regarding these principal bundles and their moduli. See for example \cite{Bea}, \cite{BG}, \cite{CH14}, \cite{CH15}, \cite{Serman}, \cite{CCH3} and many other works.

Orthogonal bundles and their moduli enjoy remarkable geometric properties. One reason for this is that the group $\SO (r , \C)$ is not simply connected as $\SL ( r, \C )$ and $\Sp ( 2n, \C )$ are, but has fundamental group $\bZ_2$ and universal covering given by the spin group $\Spin ( r , \C )$. The point of departure for the present work is the fact that for $3 \le r \le 6$, there are exceptional isomorphisms between $\Spin ( r , \C )$ and other familiar groups:
\[ \begin{array}{c||c|c|c|c}
 r & 3 & 4 & 5 & 6 \\ \hline
 \Spin ( r , \C ) & \SL ( 2, \C ) & \SL ( 2, \C) \times \SL ( 2, \C ) & \Sp ( 4, \C ) & \SL ( 4 , \C )
\end{array} \]
Moreover, in each case there are interesting geometric interpretations for the map $\Spin ( r , \C ) \to \SO ( r , \C )$; see for example \cite[Lectures 18--19]{FH}. Taking a cue from the existence of these isomorphisms, we describe the structures of orthogonal bundles of rank $r \le 6$ using sum and tensor operations on bundles of lower rank. This extends the classification in \cite[p.\ 185--187]{Mum} of $\SO ( r , \C )$-bundles for $r \le 4$. Of particular interest is the fact that topologically, orthogonal bundles of rank $r$ with trivial determinant are distinguished by another topological invariant, the second Stiefel--Whitney class $w_2 (V) \in H^2 ( C, \bZ_2 ) \cong \bZ_2$. This was characterized in \cite{Ser} in terms of theta characteristics of $C$, and we discuss it further in {\S} \ref{SW}. 
 Here is an overview of the structure theorems for the case $L = \Oc$ and $\det (V) \cong \Oc$ and $w_2 (V) = 0$. (The statements for the other cases are given at the relevant points in the paper; compare also with \cite[pp.\ 185-187]{Mum}.)
\[ \begin{array}{c|cc}
 \hbox{Rank} & \hbox{Structure} & \\ \hline
 2 & N \oplus N^{-1} : & N \in \Pic ( C ) \\
 3 & \Sym^2 E \otimes (\det E)^{-1} : & \hbox{$E$ of rank two} \\
 4 & E_1 \otimes E_2 : & \hbox{$E_1$, $E_2$ of rank two, $\det E_2 \cong (\det E_1 )^{-1}$} \\
 5 & \Ker ( \wedge^2 W_1 \to M ) : & \hbox{$M$ a square root of $\Oc$, and} \\
 & & \hbox{$W_1$ a rank four $M$-valued symplectic bundle} \\
 6 & \wedge^2 W : & \hbox{$W$ of rank four and trivial determinant}
\end{array} \]
These identifications imply in particular the existence of natural maps between certain moduli spaces of bundles over $C$, which are described in the course of the paper.

An important point is that when $r = 2n$ is even, our proofs rely on the existence of a rank $n$ isotropic subbundle in $V$. In \cite[{\S} 2]{CCH3}, it was shown that this is equivalent to the condition $\det V \cong L^n$ (in general, the isomorphism $V \isom V^* \otimes L$ only gives $(\det V)^2 \cong L^{2n}$). Moreover, in \cite{CH14} and \cite{CH15} it was shown that when $L = \Oc = \det V$, the invariant $w_2 (V)$ coincides with the parity of the degrees of the rank $n$ isotropic subbundles of $V$. In Lemma \ref{DetIrns} and Theorem \ref{w2deg} we offer more elementary proofs of these results.

Turning our attention to the ``quadric fibrations'' in the title: The locus of isotropic lines in fibers of an orthogonal bundle $V$ is a subfibration $Q_V \subset \PP V$ in smooth quadric hypersurfaces, giving $\PP V$ the structure of an ``enveloping bundle'' (see {\S} \ref{envelop}). For $r \le 6$, these quadrics are very well understood; see for example \cite{GH}:
\[ \begin{array}{c||c|c|c|c|c}
 r & 2 & 3 & 4 & 5 & 6 \\ \hline
 \hbox{Quadric} & \hbox{Two smooth points} & \hbox{Conic } \PP^1 & \PP^1 \times \PP^1 & \LG( 2, 4) & \Gr (2, 4)
\end{array} \]
Here $\LG(2, 4) \subset \Gr(2, 4)$ is the Lagrangian Grassmannian of two-dimensional subspaces of $\C^4$ which are isotropic with respect to a fixed nondegenerate skew-symmetric bilinear form. We use well-known properties of these quadrics to analyze the structures of orthogonal bundles, and in particular their isotropic subbundles. For $3 \le r \le 6$, it turns out that the isotropic Quot scheme  $\IQome$, which parameterizes isotropic subbundles of rank $m$ and degree $e$ in $V$, is isomorphic to a union of Quot schemes of  subbundles of certain bundles of lower rank. This allows us in particular to enumerate the maximal isotropic subbundles of a general orthogonal bundle of rank $r \le 6$, when the number is finite. This gives an analog for low rank orthogonal bundles of the enumeration results in \cite{Hol} for maximal degree subbundles of a general vector bundle, and in \cite{CCH2} for maximal degree Lagrangian subbundles of a general symplectic bundle.

Some results in the paper are proven in or would follow easily from \cite{CCH3}, \cite{CH14} and \cite{CH15}. However, our policy is, whenever possible, to give arguments which are elementary and/or use the rich and familiar geometry of the low rank situation.

The paper is organized as follows. We begin by recalling background material on orthogonal bundles and their moduli, isotropic subbundles and isotropic Quot schemes, and quadric subfibrations of projective bundles. Then we give a detailed treatment of the rank two case. This is relatively well known and straightforward, but provides a paradigm for the other cases; moreover, the rank two case has some useful corollaries for higher rank bundles. We then turn to the higher rank cases. It emerges that the results in rank three and five follow naturally from special cases of rank four and six respectively. Therefore, we treat the latter first in each case.

For the present, as mentioned above, we have restricted ourselves in the even rank cases to orthogonal bundles which admit isotropic subbundles of half rank. The structure of bundles not possessing any such subbundle is of interest, and amenable to study by similar methods. This topic will be studied in a future project.

\subsection*{Acknowledgements} We are grateful to Daewoong Cheong for stimulating conversations, and for sharing valuable knowledge of enumerative geometry. We thank Han-Bom Moon for a useful suggestion on the morphisms between moduli spaces. The second author thanks Raquel Mallavibarrena for interesting communication about enveloping bundles.

\section{Preliminaries}

\subsection{Orthogonal bundles and isotropic subbundles} \label{intro}

Here we recall background material on orthogonal bundles. For more information, we refer to \cite{Bea}, \cite{Mum}, \cite{Ram} and \cite{CCH3}.

Let $C$ be a complex projective smooth curve, and $L \to C$ a line bundle. A vector bundle $V \to C$ of rank $r$ is said to be \textsl{$L$-valued orthogonal} if there is a nondegenerate symmetric bilinear form $\sigma \colon V^{\otimes 2} \to L$; equivalently, if there is a symmetric isomorphism $V \isom V^*\otimes L$.

In particular, for such a $V$ we have $(\det V)^2 \cong L^r$. It follows that if $r = 2n+1$ is odd,
\begin{multline} \label{DetOddRank} \deg (L) \hbox{ is even, and} \det V \cong M^{2n+1} \hbox{ for a uniquely determined $M$} \\ \hbox{such that } M^2 \cong L . \end{multline}
On the other hand, if $r = 2n$ is even, then
\begin{equation} \label{DetVEta} \det V \ \cong \ L^n \otimes \eta \hbox{ for some $\eta$ satisfying } \eta^{ 2} \ \cong \ \Oc . \end{equation}
The question of which $\eta$ arise in (\ref{DetVEta}) is linked to the existence of certain isotropic subbundles, which we now discuss.

Let $V$ be $L$-valued orthogonal of rank $r$. For any subbundle $E \subset V$, the isomorphism $V \isom V^* \otimes L$ induces a sequence $0 \to E^\perp \to V \to E^* \otimes L \to 0$, where
\[ E^\perp \ := \ \{ v \in V : \sigma ( v \otimes e ) = 0 \hbox{ for all } e \in E \} \]
is the orthogonal complement of $E$ with respect to $\sigma$. A subbundle $E \subset V$ is \textsl{isotropic} if $E \subseteq E^\perp$; equivalently, if $\sigma ( E \otimes E ) = 0$.

Suppose $\rank (V) = 2n$ and $V$ contains an isotropic subbundle $E$ of rank $n$. Then $E = E^\perp$ and $V$ is an extension $0 \to E \to V \to E^* \otimes L \to 0$. In particular,
\begin{equation} \label{IrnsDet} \hbox{if $V$ admits a rank $n$ isotropic subbundle, then $\det V \cong L^n$;} \end{equation}
that is, $\eta = \Oc$ in (\ref{DetVEta}). In fact, by \cite[Lemma 2.5]{CCH3} the converse of (\ref{IrnsDet}) also holds; in Proposition \ref{DetIrns}, we shall offer an elementary proof of this fact.

Moreover, in general, not every extension $0 \to E \to V \to E^* \otimes L \to 0$ arises from an orthogonal structure. We shall use the following characterization \cite[Criterion 2.1]{Hit}.

\begin{criterion} \label{extension} Let $0 \to E \to V \to E^* \otimes L \to 0$ be an extension of vector bundles, with class $\delta(V) \in H^1 ( C, E \otimes E \otimes L^{-1} )$. Then $V$ admits an orthogonal structure with respect to which $E$ is isotropic if and only if, up to the action of $\Aut (E) \times \Aut (E^* \otimes L)$, the extension class $\delta(V)$ belongs to $H^1 ( C, \wedge^2 E \otimes L^{-1} )$. \end{criterion}

\subsection{Moduli of orthogonal bundles} 

For any line bundle $N$, let $\SU (r, N )$ be the moduli space of S-equivalence classes of semistable vector bundles of rank $r$ and determinant $N$ over $C$.  As this is noncanonically isomorphic to the moduli space of semistable principal $\SL ( 2, \C )$-bundles over $C$, by \cite[Theorem 5.9]{Rth} it is an irreducible projective variety of dimension $(r^2 - 1)(g-1)$. We consider the following closed subloci of $\SU ( r, \Oc )$ and $\SU ( r, \Oc (nx) )$ for a fixed $x \in C$.
\begin{itemize}
\item $\MO (2n, \Oc)$, the moduli space of semistable $\Oc$-valued orthogonal bundles of rank $2n$ and trivial determinant
\item $\MO (2n, \Oc(x))$, the moduli space of semistable $\Oc(x)$-valued orthogonal bundles of rank $2n$ and determinant $\Oc(nx)$
\item $\MO (2n+1, \Oc)$, the moduli space of semistable $\Oc$-valued orthogonal bundles of rank $2n+1$ and trivial determinant
\end{itemize}
We shall use the notation $\SU ( r, N )_\st$ and $\MO ( r, L )_\st$ when we wish to restrict to the stable loci. We write $F \sim F'$ to indicate that $F$ and $F'$ are S-equivalent semistable vector bundles.

Let us explain why we consider only these three kinds of moduli space. Firstly, when $r = 2n+1$, any $L$-valued orthogonal bundle is a twist by a line bundle of one which is $\Oc$-valued with trivial determinant. When $r = 2n$, any $L$-valued orthogonal bundle is a twist by a line bundle of either an $\Oc$-valued or an $\Oc (x)$-valued orthogonal bundle. Furthermore,  we shall consider only $L$-valued orthogonal bundles which admit rank $n$ isotropic subbundles. Hence in view of (\ref{IrnsDet}), we get the restriction on the determinant as above for $r = 2n$.

The following seems to be well known, but we include a proof for completeness.

\begin{lemma} Every component of the moduli spaces above is of dimension $\frac{1}{2}r(r-1)(g-1)$, where $r$ is the rank. \end{lemma}

\begin{proof} Recall that the \textsl{conformal orthogonal group} $\GO ( r, \C )$ is the image in $\GL ( r, \C )$ of the multiplication map $\mathrm{O} ( r, \C ) \times \C^* \to \GL ( r, \C )$. As explained in \cite[{\S} 2]{BG}, the moduli space of rank $2n$ vector bundles admitting an $L$-valued symmetric form for some $L$ of degree $\ell$ is the image of certain components of the moduli space of semistable $\GO(2n, \C)$-bundles by the forgetful map to $\cU \left( r, \frac{r \ell}{2} \right)$. Furthermore, by \cite[Theorem 8.5 and Remark 8.6]{BS}, this map is finite. Thus by \cite[Theorem 5.9]{Rth}, this image has dimension
\[ (g - 1 ) \cdot \dim \GO ( r, \C ) + \dim Z ( \GO ( r, \C ) ) \ = \ \left( \frac{1}{2}r(r-1) + 1 \right) (g-1) + 1 . \]
Fixing $\ell$ and $L$, we obtain a locus of codimension $g$; that is, dimension $\frac{1}{2}r(r-1)(g-1)$ as desired. \end{proof}

\subsection{The second Stiefel--Whitney class} \label{SW}

As will be discussed later, the moduli space $\MO ( r, \Oc(x) )$ is irreducible. However, $\MO (r, \Oc )$ is in general reducible, due to the presence of another discrete invariant for orthogonal bundles which we shall now discuss.

Firstly, we dispose of a special case. The group $\SO (2, \C)$ is isomorphic to $\C^*$ via the map
\[ z \ \mapsto \ \begin{pmatrix} \frac{z + 1/z}{2} & \frac{-z + 1/z}{2 \sqrt{-1}} \\ \frac{z - 1/z}{2 \sqrt{-1}} & \frac{z + 1/z}{2} \end{pmatrix} . \]
Hence it has infinite cyclic fundamental group, unlike $\SO ( r, \C )$ for $r \ge 3$. Therefore, during this section we shall assume that $r \ge 3$.

For $r \ge 3$, let $1 \to \bZ_2 \to \Spin ( r , \C ) \to \SO ( r , \C ) \to 1$ be the exact sequence associated to the universal covering of $\SO ( r , \C )$; see for example \cite[Lecture 20]{FH}. By for example \cite[Chapter 5]{Gro58}, this gives rise to a sequence of nonabelian cohomology sets
\[ \cdots \ \to \ H^1 ( C, \bZ_2 ) \ \to \ H^1 ( C, \Spin ( r , \C ) ) \ \to \ H^1 ( C, \SO ( r , \C ) ) \ \to \ H^2 ( C, \bZ_2 ) . \]
Choosing a set of transition functions for $V$ defines an element of $H^1 ( C, \SO ( r , \C ) )$. Then $w_2 (V)$ is the image of this element in $H^2 ( C, \bZ_2 )$. By \cite[Proposition 5.7.2]{Gro58}, the class $w_2 (V)$ is trivial if and only if $V$ is obtained from a $\Spin ( r , \C )$-bundle by the extension of structure group given by $\Spin ( r , \C ) \to \SO ( r , \C )$. For $i \in \bZ_2$, we denote by $\MO^i ( r, \Oc )$ the component of $\MO ( r, \Oc )$ consisting of bundles with $w_2 ( V ) = i \in \{0,1\}$.

The class $w_2(V)$ has several properties. Firstly, by \cite[Theorem 2]{Ser}, for every theta characteristic $\kappa$ on $C$ we have
\begin{equation}
 \label{Serre} w_2 (V) \ \equiv \ h^0 ( V \otimes \kappa ) + r \cdot h^0 ( \kappa ) \mod 2 .
\end{equation}

Furthermore, it is shown in \cite{CH14} and \cite{CH15} that $w_2 (V)$ coincides with the parity of the degree of any rank $n$ isotropic subbundle of $V$, where $n = \left\lfloor \frac{\rank (V)}{2} \right\rfloor$. We now offer a more elementary proof of this fact.

\begin{theorem} \label{w2deg} Let $V$ be an $\Oc$-valued orthogonal bundle of rank $2n$ or $2n+1$ and trivial determinant. If $F \subset V$ is a rank $n$ isotropic subbundle, then $w_2 (V) \equiv \deg (F) \mod 2$. In particular, all rank $n$ isotropic subbundles of $V$ have degree of the same parity. \end{theorem}

\begin{caution} Note that we have not proven that $V$ has any such subbundle $F$. This will be shown in Proposition \ref{DetIrns}. \end{caution}

\begin{proof} Suppose firstly that $r = 2n$. By hypothesis and using Criterion \ref{extension}, there is an exact sequence $0 \to F \to V \to F^* \to 0$ such that the cohomology class $[V]$ belongs to $H^1 (C, \wedge^2 F )$. For any theta characteristic $\kappa$, this yields a cohomology sequence
\begin{equation} 0 \ \to \ H^0 ( C, F \otimes \kappa ) \ \to \ H^0 ( C, V \otimes \kappa ) \ \to \ H^0 (C,  F^* \otimes \kappa ) \ \xrightarrow{\cup [V]} \ H^1 (C,  F \otimes \kappa ) \ \to \ \cdots . \label{kappaseq} \end{equation}
Now by Serre duality, there is an identification $H^1 ( C, F \otimes \kappa ) \isom H^0 ( C, F^* \otimes \kappa )^*$. We claim that the map $\cup [V]$ is antisymmetric. To see this: It is well known that the cup product map $\cup \colon H^1 (C,  F \otimes F ) \to \Hom \left( H^0 (C,  F^* \otimes \kappa ), H^1 (C,  F \otimes \kappa ) \right)$ is dual to the Petri map
\[ \mu \colon H^0 (C,  F^* \otimes \kappa )^{\otimes 2} \ \to \ H^0 (C,  F^* \otimes F^* \otimes K_C ) . \]
Thus we must show that
\begin{equation} \label{annih} \left\langle [V] , \mu \left( \Sym^2 H^0 (C,  F^* \otimes \kappa ) \right) \right\rangle \ = \ 0 , \end{equation}
the pairing being defined by Serre duality. But since
\[ \mu \left( \Sym^2 H^0 ( C, F^* \otimes \kappa ) \right) \ \subseteq \ H^0 (C,  \Sym^2 F^* \otimes \Kc ) \]
while $[V] \in H^1 (C,  \wedge^2 F ) \cong H^0 (C,  \wedge^2 F^* \otimes \Kc )^*$, we obtain (\ref{annih}).


By the above claim, the map $\cup [V]$ in (\ref{kappaseq}) has even rank. We write
\[ h \ := \ h^0 (C,  F^* \otimes \kappa) \ = \ h^1 (C,  F \otimes \kappa ) . \]
Moreover, $\chi ( F \otimes \kappa ) 
 = \deg (F)$, so $h^0 (C,  F \otimes \kappa ) = \deg (F) + h$. By exactness of (\ref{kappaseq}), then,
\[ h^0 ( C, V \otimes \kappa ) \ = \ h^0 (C,  F \otimes \kappa) + \dim \Ker ( \cup [V] ) \ = \ \left( \deg(F) + h \right) + \left( h - \rank ( \cup [V] ) \right) . \]
Since $\rank ( \cup [V] )$ is even, $h^0 ( C, V \otimes \kappa ) \equiv \deg(F) \mod 2$. The statement now follows from (\ref{Serre}), since $r = 2n$.

Suppose now that $\rank(V) = 2n+1$. We use a technique applied frequently in \cite{CH15} and \cite[{\S} 8]{CCH3}, which will appear several times in this paper. The orthogonal direct sum $V \perp \Oc$ is $\Oc$-valued orthogonal of rank $2n+2$ and trivial determinant. Let $E$ be any rank $n$ isotropic subbundle of $V$. Using for example a local trivialization, we can complete $E$ to a rank $n+1$ isotropic subbundle $\bE$ of $V$ in two distinct ways, each one an extension $0 \to E \to \bE \to \Oc \to 0$. Thus $\deg (E) = \deg ( \bE )$. It follows by the even rank case above that
\[ \hbox{all rank $n$ isotropic subbundles of $V$ have parity equal to $w_2 (V \perp \Oc)$.} \]
On the other hand, an easy computation with (\ref{Serre}) shows that $w_2 ( V ) = w_2 ( V \perp \Oc )$. 
 The statement for $r = 2n+1$ now follows. \end{proof}

%

\subsection{Isotropic Quot schemes}

For any bundle $W$, we have the Quot scheme
\[ \Quot^{\rank(W) - m, \deg(W) - d} (W) \]
parameterizing coherent quotients of $W$ of rank $\rank(W) - m$ and degree $\deg(W) - d$; equivalently, locally free subsheaves of $W$ of rank $m$ and degree $d$. We depart from convention and denote this by $\Quot_{m, d} (W)$. We denote the open subset of saturated subsheaves by $\Quoto_{m, d} (W)$.

If $V$ is an orthogonal bundle of rank $2n$ or $2n+1$, then  for $1 \le m \le n$ we consider the closed sublocus
\[ \IQome \ := \ \left\{ [ j \colon E \hookrightarrow V ] \in \Quoto_{m, e} (V) : \hbox{$j(E)$ isotropic of rank $m$ and degree $e$} \right\} . \]
Let us say a little more about isotropic subbundles of maximal rank $n$. By \cite[Proposition 3.3]{CCH3}, the expected dimension of $\IQoe$ is
\begin{equation} \label{Inle}
\begin{cases}
 -(n-1)e - \frac{n(n-1)}{2} ( g - 1 - \ell ) & \hbox{ if } r = 2n ; \\
 -n \left( e + \frac{\ell}{2} \right) - \frac{n(n+1)}{2} ( g - 1 - \ell ) & \hbox{ if } r = 2n + 1 .
\end{cases} 
\end{equation}

\noindent Note that when $r = 2n$ and $\ell$ is even, $\IQoe$ is in general disconnected, by \cite[Theorem 1.4 (a)]{CCH3}. In Corollary \ref{rk4IQ} and Corollary \ref{rk6IsotSubbsCor}, we shall see this concretely for $r = 4$ and $r = 6$ respectively.

\subsubsection*{Zero dimensional Quot schemes} \label{ZeroDimIQ}

Of particular interest for us will be the question of enumerating the points of $\IQoe$ for a general $V$ when this scheme is finite. This corresponds to the maximal value of $e$ for which $\IQoe$ is nonempty. The analogous question for maximal subbundles of vector bundles was answered in \cite[Theorem 4.2]{Hol}, and for maximal Lagrangian subbundles of symplectic bundles in \cite{CCH2}. We quote a special case \cite[Corollary 4.3]{Hol} which we shall use later (see also \cite[Lemma 2.2]{Ox}).

\begin{theorem} \label{MaxLineSubbs} Let $W$ be a general stable bundle of rank $r$. If $\deg (W) \equiv 1-g \mod r$ then $W$ has $r^g$ maximal line subbundles of degree $\frac{1}{r} \left(\deg (W) - (r-1)(g-1) \right)$. \end{theorem}

%
%
%
Using (\ref{Inle}), we now decide which $\IQoe$ have expected dimension zero. The following can be viewed as a special case of \cite[Theorem 1.4 and 1.5]{CCH3}.

\begin{lemma} \label{zerodim} 
Write $\ell := \deg (L)$. We assume that $\ell \in \{ 0 , \frac{r}{2} \}$ when $r$ is even, and $\ell = 0$ when $r$ is odd. Assume that $V$ is $L$-valued orthogonal of rank $r$ and general in moduli. In the range $3 \le r \le 6$, the scheme $\IQoe$  has expected dimension zero
 if the following conditions hold:
\[
\begin{array}{|r|c|} \hline 
r = 3 & e = - g + 1 \\ \hline
r = 4 & e = \ell - g + 1 \\ \hline
r = 5 & g \text{ is odd and } e =  -\frac{3}{2}(g-1) \\ \hline
r = 6 & \ell - g \text{ is odd and } e = \frac{3}{2}(\ell - g + 1) \\ \hline
\end{array}
\]
\end{lemma}
\noindent In these cases, each component (that is, point) of $\IQoe$ corresponds to an isotropic subbundle of $V$ of rank $\lfloor \frac{r}{2} \rfloor$ and degree $e$. Later we shall give enumerative results on the number of such isotropic subbundles of rank $m$, for $1 \le m \le n$.

\subsection{Quadric fibrations} \label{envelop}

Finally, we discuss a certain quadric subfibration $Q_V \subset \PP V$ for an orthogonal bundle $V$. This is closely related to the notion of an enveloping bundle as studied in \cite{LM} and \cite{LMP}, to which we refer the reader for more information. Note however that our situation is slightly different, as we begin with the vector bundle $V$ instead of the quadric fibration $Q_V$; and we do not consider $\PP V$ or $Q_V$ in a projective embedding, but only $Q_V$ relatively embedded in $\PP V$ over $C$.

Let $V \to C$ be a vector bundle of rank $r \ge 2$, and $\pi \colon \PP V \to C$ the associated projective bundle. Let $\sigma$ be a nonzero element of $H^0 ( C, \Sym^2 V^* \otimes L )\cong H^0 ( \PP V , \cO_{\PP V} (2) \otimes \pi^* L )$. We denote by $Q_V$ the divisor $( \sigma ) \subset \PP V$. It is a quadric fibration over $C$. If we further assume that $\sigma|_p$ defines a symmetric isomorphism $V|_p \to (V^*\otimes L)|_p$ for each fiber, then $\sigma$ defines an $L$-valued orthogonal structure on $V$, and $Q_V \to C$ is a fibration in \emph{smooth} quadrics. (This was observed for $r = 3$ and $C = \PP^1$ in \cite[Remark 1, p.\ 9]{LM}.) In this case, $Q_V \subset \PP V$ parameterizes lines in fibers of $V$ which are isotropic with respect to the quadratic form $\sigma$.\\
\\
A useful application of the above is:

\begin{proposition} \label{IsotLineSubb} Let $V$ be an $L$-valued orthogonal bundle of rank $r \ge 3$. Then $V$ admits an isotropic line subbundle. \end{proposition}

\begin{proof} The fibers of $Q_V$ are quadric hypersurfaces in $\PP^{r-1}$, so are rationally connected for $r \ge 3$. Thus by \cite{GHS}, there is a section $C \to Q_V \cong \OG (1, V)$, defining an isotropic line subbundle of $V$. \end{proof}

\section{Rank two}

We now consider orthogonal bundles of rank two. These results are straightforward and can be obtained without mentioning quadric fibrations, but we do this as practice for the higher rank cases. Firstly, we give a variation on \cite[Proposition 2.10]{CCH3}.

\begin{theorem} \label{rk2Structure}  Let $V$ be an  $L$-valued orthogonal bundle of rank two and determinant $L$. Then 
$V \cong N \oplus N^{-1} L$ 
 for some line bundle $N$.  Moreover, $V$ has precisely two isotropic subbundles, which are isomorphic to $N$ and $N^{-1}L$ respectively. \end{theorem}

\begin{proof} 
Let $N$ be a line subbundle of maximal degree in $V$. Then the orthogonal form on $V$ induces the surjective bundle map $ V \to N^{-1} L  $ whose kernel  is $N^\perp$. Thus we have $N^\perp \cong N$.

If $N$ is isotropic, Criterion \ref{extension} shows that $V$ is isomorphic to an extension $0 \to N \to V \to N^{-1} L \to 0$ determined by a class in $H^1 ( C, L^{-1} \otimes \wedge^2 N )$. As $\wedge^2 N = 0$, we have $V \cong N \oplus N^{-1} L$.

If $N$ is not isotropic,  then
$N \oplus N^\perp $ is a rank two subsheaf of $V$.  Thus in fact $N \oplus N$ is a subsheaf of $V$. As $\deg (N)$ was assumed to be maximal, this does not drop rank anywhere, so $V \cong N \oplus N$. Since $\det V \cong L$, we have $N \cong N^{-1} L$. In particular, $E \cong N \otimes \cO_C^{\oplus 2}$, where $N$ is an orthogonal line subbundle. Hence by choosing a copy of $N \subset V$ whose fiber is isotropic, we can also write $E \cong N \oplus N^{-1}L$, with respect to which the summands are the isotropic subbundles.
 \end{proof}

We now investigate the structure of $\PP V$ as an enveloping bundle, where $V$ is $L$-valued orthogonal of rank two and determinant $L$. Here the locus $Q_V$ of isotropic lines in $\PP V$ 
 is a fibration in smooth zero-dimensional quadrics; that is, pairs of points. By Theorem \ref{rk2Structure}, we have
\[ Q_V \ = \ \PP N \sqcup \PP ( N^{-1} L ) \ \cong \ C \sqcup C . \]
An isotropic line subbundle corresponds to a section of $\PP V$ belonging to $C \sqcup C$ at each point; that is, one of the copies of $C$. Thus $\IQ_{1,e} (V)^\circ$ is empty unless $e = \deg (N)$ or $e = \deg (N^{-1} L)$. (If $e = \deg (N) = \frac{\ell}{2}$ then it consists of two points.)\\
\par
As applications of our results so far, and to make the present paper more self-contained, we offer alternative proofs of some results in \cite{CCH3} and \cite{CH15} which we shall require later.

\begin{proposition} \label{DetIrns}Let $V$ be an $L$-valued orthogonal bundle of rank $2n$. Then $V$ admits an isotropic subbundle of rank $n$ if and only if $\det V \cong L^n$. \end{proposition}

\begin{proof} (See also \cite[Lemma 2.5]{CCH3}.) The ``only if'' implication is (\ref{IrnsDet}). For ``if'', we proceed by induction on $n$. The case $n = 1$ follows from Theorem \ref{rk2Structure}. Suppose $n \ge 2$. By Proposition \ref{IsotLineSubb}, we can find an isotropic line subbundle $H \subset V$. Then $H^\perp / H$ is $L$-valued orthogonal of rank $2n - 2$. From the exact sequence $0 \to H^\perp \to V \to H^{-1} L \to 0$, we deduce that
\[ \det \left( H^\perp / H \right) \ \cong \ \det ( H^\perp ) \otimes H^{-1} \ \cong \ ( \det V ) \otimes H \otimes L^{-1} \otimes H^{-1} \ \cong \ (\det V) \otimes L^{-1} \ \cong \ L^{n-1} . \]
By induction, $H^\perp / H$ has an isotropic subbundle $\bE$ of rank $n - 1$. The inverse image of $\bar{E}$ in $H^\perp \subset V$ is a rank $n$ isotropic subbundle of $V$. \end{proof}

We now give an analog of Theorem \ref{w2deg} for the odd degree case (compare with \cite[Theorem 1.4 (b)]{CCH3}).

\begin{lemma} \label{EvenAndOddDeg} Suppose $\deg (L)$ is odd, and let $V$ be $L$-valued orthogonal of rank $2n$. Suppose $V$ satisfies the equivalent conditions of Proposition \ref{DetIrns}. Then $V$ has isotropic rank $n$ subbundles both of even degree and of odd degree. \end{lemma}

\begin{proof} Let $E \subset V$ be a rank $n$ isotropic subbundle. Let $H \subset E$ be any rank $n-1$ subbundle, which is necessarily isotropic. Then $H^\perp / H$ is $L$-valued orthogonal of rank two and determinant $L$, and admits $E/H$ as an isotropic line subbundle. By Theorem \ref{rk2Structure}, in fact
\[ \frac{H^\perp}{H} \ \cong \ \frac{E}{H} \oplus \left( \left( \frac{E}{H} \right)^* \otimes L \right) . \]
Then $(E/H)^* \otimes L$ is of the form $F/H$ for a rank $n$ isotropic subbundle $F \subset H^\perp \subset V$, and
\[ \deg (F) \ = \ - (\deg (E) - \deg H) + \deg (L) + \deg (H) \ \equiv \ \deg (E) + \deg (L) \mod 2 . \]
As $\deg(L)$ is assumed to be odd, $\deg (F)$ and $\deg (E)$ have opposite parity. \end{proof}

\section{Rank four}

We now turn our attention to orthogonal bundles of rank four, as the rank three case turns out to follow from a special case of rank four.

\subsection{Rank four orthogonal bundles with rank two isotropic subbundles}

As noted in \cite[p.\ 186]{Mum}, given two bundles $E_1$ and $E_2$ of rank two, the pairing
\begin{equation} (E_1 \otimes E_2) \otimes (E_1 \otimes E_2) \ \to \ \wedge^2 E_1 \otimes \wedge^2 E_2 \ =: \ L \label{rk4form} \end{equation}
defines an $L$-valued orthogonal structure on $E_1 \otimes E_2$. Furthermore, for any line subbundles $N_1 \subset E_1$ and $N_2 \subset E_2$, the subbundles $N_1 \otimes E_2$ and $E_1 \otimes N_2$ are isotropic of rank two. We shall now see that this induces finite surjective morphisms of moduli spaces.

\begin{theorem} \label{rk4moduli} \quad
\begin{enumerate}
\item[(a)] The association $( F_1, F_2 ) \mapsto F_1 \otimes F_2$ defines finite surjective morphisms
\[
 \Theta_4^0 \colon \SU ( 2, \Oc) \times \SU (2, \Oc) \ \to \ \MO^0 (4, \Oc)
\]
and 
\[
 \Theta_4^1 \colon \SU ( 2, \Oc(x) ) \times \SU (2, \Oc(-x)) \ \to \ \MO^1 (4, \Oc).
\]
In particular, $\MO (4, \Oc)$ has two irreducible components.

\item[(b)] For $L \cong \Oc(x)$: The association $( F_1, F_2 ) \mapsto F_1 \otimes F_2$ defines a finite  surjective morphism
\[
\Phi_4 \colon \SU ( 2, \Oc) \times \SU (2, \Oc(x)) \to \MO ( 4 , \Oc(x) )
\]
In particular, $\MO (4, \Oc(x))$ is irreducible.
\end{enumerate}
\end{theorem}

\begin{proof} Let us firstly define the morphism $\Theta_4^0$. As $\C$ has characteristic zero, $F_1 \otimes F_2$ is semistable if $F_1$ and $F_2$ are so. Therefore, passing to a suitable \'etale cover, we can define $\Theta_4^0$ on the open subset $\SU ( 2, \Oc)_\st \times \SU (2, \Oc)_\st$ using a Poincar\'e bundle. Note that by Theorem \ref{w2deg}, the images of $\Theta_4^0$ and $\Theta_4^1$ are contained in $\MO^0 (4, \Oc)$ and $\MO^1 (4, \Oc)$, respectively.

Now it is straightforward to check that if $F_1 \sim F_1'$ and $F_2 \sim F_2'$, then $F_1 \otimes F_2 \sim F_1' \otimes F_2'$. Hence, as any semistable bundle is a limit of stable ones, it follows that $\Theta_4^0$ has a continuous extension to $\SU ( 2, \Oc) \times \SU (2, \Oc)$. Now since $\Theta_4^0$ is algebraic on $\SU ( 2, \Oc)_\st \times \SU (2, \Oc)_\st$, it is bounded on a neighborhood of any point. Therefore, by the Riemann extension theorem, the continuous extension to  $\SU ( 2, \Oc) \times \SU (2, \Oc)$ given by the association $( F_1, F_2 ) \mapsto F_1 \otimes F_2$ is   a morphism.

Now by \cite[Th\'eor\`eme B (b)]{DN}, the moduli space $\SU ( 2 , \Oc )$ has Picard number 1. Let $\cN$ be any very ample line bundle on $\MO^0 ( 4, \Oc )$. Then $\left( \Theta_4^0 \right)^* \cN \cong \cL^a \boxtimes \cL^b$ for some integers $a, b$, where $\cL$ is the ample generator of $\Pic \left( \SU ( 2, \Oc ) \right)$. Since $\Theta_4^0$ is symmetric, we have $a = b > 0$. Therefore, $\Theta_4^0$ must be a finite morphism. As $\SU ( 2, \Oc) \times \SU (2, \Oc)$ and $\MO^0 (4, \Oc)$ are projective varieties of the same dimension,  $\Theta_4^0$ is surjective.

Similar arguments show that $\Theta_4^1$ and $\Phi_4$ are finite surjective morphisms. (Note that since $\SU ( 2, \Oc ( \pm x ))$ is a fine moduli space, we can define $\Theta_4^1$ without using the Riemann extension theorem.)  
The irreducibility statements follow from the fact that $\SU ( 2, L )$ is irreducible. \end{proof}

It is natural to ask if an arbitrary (possibly unstable) $L$-valued orthogonal bundle of rank four which admits an isotropic subbbundle of rank two can be obtained in this way. The  answer is Yes, and the proof of the following theorem shows how to find the expression(s).

\begin{theorem} \label{rk4structure} Let $V$ be a rank four $L$-valued orthogonal bundle. Then $V$ admits a rank two isotropic subbundle 
if and only if $V$ is isomorphic to $E_1 \otimes E_2$ for some rank two bundle $E_2$ with $\det E_2 \cong \det E_1^* \otimes L$. \end{theorem}

\begin{proof} One implication follows from  (\ref{rk4form}) and the discussion after it. Conversely, suppose $E_1 \subset V$ is a rank two isotropic subbundle. Then Criterion \ref{extension} shows that $V$ is isomorphic to an extension $0 \to E_1 \to V \to E_1^* \otimes L \to 0$ defined by a class
\[ \delta \ \in \ H^1 (C,  \wedge^2 E_1 \otimes L^*) \ \subset \ H^1 (C,  E_1 \otimes E_1 \otimes L^*) . \]
Now $\delta$ also defines an extension of line bundles $0 \to \Oc \to E_2 \to \det E_1^* \otimes L \to 0$. Tensoring by $E_1$, we obtain an extension 
$0 \to E_1 \to E_1 \otimes E_2 \to E_1 \otimes \det E_1^*  \otimes L \to 0$ of class
\[ \Iden_{E_1} \otimes \delta \ \in \ H^1 (C,  (\Oc \cdot \Iden_{E_1} ) \otimes \det E_1 \otimes L^*) \ \subset \ H^1 (C,  \End ( E_1 ) \otimes \det E_1 \otimes L^*) . \]
Now since $E_1$ has rank two, there is an isomorphism $\Hom ( E_1 , E_1 \otimes \det E_1 ) \isom E_1 \otimes E_1$ inducing an isomorphism of split exact sequences
\[ \xymatrix{ 0 \ar[r] & (\Oc \cdot \Iden_{E_1}) \otimes \det E_1 \ar[r] \ar[d]^\wr & \End (E_1) \otimes \det E_1 \ar[r] \ar[d]^\wr & \End_0 (E_1) \otimes \det E_1 \ar[r] \ar[d]^\wr & 0 \\
0 \ar[r] & \wedge^2 E_1 \ar[r] & E_1 \otimes E_1 \ar[r] & \Sym^2 E_1 \ar[r] & 0 . } \]
Under the induced identification $H^1 ( C, \wedge^2 E_1 \otimes L^*) \isom H^1 (C,  \Oc \cdot \Iden_{E_1} \otimes \det E_1  \otimes L^* )$, the extension class $\Iden_{E_1} \otimes \delta$ corresponds simply to $\delta$. This shows that $V$ and $E_1 \otimes E_2$ are isomorphic vector bundles. Also the natural orthogonal structure of $V$ induced from the extension $\delta$ coincides with that of $E_1 \otimes E_2$ in view of Criterion \ref{extension}. \end{proof}

\subsection{Quadric fibration and isotropic subbundles}

Let $V = E_1 \otimes E_2$ be an $L$-valued orthogonal bundle as above. A computation using (\ref{rk4form}) shows that the isotropic vectors in $E_1 \otimes E_2$ are exactly those of the form $e_1 \otimes e_2$ for $e_i \in E_i$. 
 Therefore:

\begin{proposition} \label{rk4envelop} The quadric subfibration $Q_V \subset \PP V$ parameterizing isotropic lines is the image of the relative Veronese embedding  $\PP E_1 \times_C \PP E_2 \ \hookrightarrow \ \PP ( E_1 \otimes E_2 )$. \end{proposition}

In particular, $\PP V$ is an enveloping bundle for $\PP E_1 \times_C \PP E_2$. Moreover, we observe:

\begin{remark} \label{rk4decomp} \quad
\begin{enumerate}
\item[(a)] There are natural projections $\rho_1 \colon Q_V \to \PP E_1$ and $\rho_2 \colon Q_V \to \PP E_2$, which are morphisms over $C$. \label{Qprojs}  
\item[(b)] Via the identification $E_1 \otimes E_2 \isom \Hom ( E_1^* , E_2 )$, the locus $Q_V$ is the projectivization of the set of maps of rank one.
\item[(c)] The two-dimensional isotropic subspaces of a fiber $E_1 \otimes E_2|_p$ are exactly those of the form $\Lambda_1 \otimes E_2|_p$ and $E_1|_p \otimes \Lambda_2$, where $\Lambda_i$ is a line in $E_i|_p$. 
\end{enumerate}
\end{remark}
We now use the enveloping bundle structure to study isotropic subbundles of $V$.

\begin{proposition} \label{rk4IsotSubbs} Suppose $V = E_1 \otimes E_2$ as above. Let $F \subset V$ be a rank two isotropic subbundle. Then $F$ is either of the form $E_1 \otimes N_2$ for a line bundle $N_2 \subset E_2$, or $N_1 \otimes E_2$ for a line bundle $N_1 \subset E_1$. In each case, $N_i$ is uniquely determined as a subbundle of $E_i$. \end{proposition}

\begin{proof} The inclusion $F \hookrightarrow V$ induces a map $\PP F \hookrightarrow \PP V$ of ruled varieties over $C$. By isotropy, $\PP F$ is contained in $Q_V$, so we have the composed maps $\PP F \hookrightarrow Q_V \xrightarrow{\rho_i} \PP E_i$. Clearly at least one of these is dominant, and hence surjective, being a morphism of projective varieties. Moreover, since $F$ is isotropic, by Remark \ref{rk4decomp} (c), for all $p \in C$ the fiber $F|_p$ is of the form $\Lambda_1 \otimes E_2|_p$ or $E_1|_p \otimes \Lambda_2$ for some line $\Lambda_i \subset E_i|_p$. Thus $\PP F$ cannot project surjectively to both $\PP E_1$ and $\PP E_2$, so exactly one of the projections $\PP F \to \PP E_i$ is surjective.

After reordering if necessary, we may assume that $\PP F \cong \PP E_1$. As $\PP F$ is not dominant over $\PP E_2$, for general $p \in C$ the fiber $\PP F|_p \subset Q_V|_p$ is the line
\[ \PP E_1|_p \times \{ \mu ( p ) \} \ \subset \ \PP E_1|_p \times \PP E_2|_p \]
for some $\mu(p) \in \PP E_2|_p$. We thus obtain a section $\mu \colon C \to \PP E_2$, defining a line subbundle $N_2 \subset E_2$, and $F = E_1 \otimes N_2$. \end{proof}

\begin{remark} If $\PP E_1 \cong \PP E_2$ then $E_1 \otimes E_2$ is a twist of $\End\, E_1 = \Oc \oplus \End_0 \, E_1$, and in particular is strictly semistable. Thus if $V = E_1 \otimes E_2$ is a stable vector bundle, then $\PP E_1 \not\cong \PP E_2$. The situation where $\PP E_1 \cong \PP E_2$ will arise in the next section in the context of rank three orthogonal bundles. \end{remark}

\noindent In what follows, we follow the convention that a Quot scheme $\Quot_{r, d} (W)$ is empty if $d$ is not an integer.

\begin{corollary} \label{rk4IQ} Let $V = E_1 \otimes E_2$ be a rank four $L$-valued orthogonal bundle as above. Write $d_i := \deg (E_i)$. For each $e$, there is an isomorphism
\[ \Quoto_{1, \frac{1}{2}(e-d_2)} (E_1) \ \sqcup \ \Quoto_{1, \frac{1}{2}(e-d_1)} (E_2) \ \isom \ \IQo_{2, e} ( V ) , \]
via the map
\[ \begin{cases} N \mapsto N \otimes E_2 \hbox{ if } N \subset E_1 \\
 N \mapsto E_1 \otimes N \hbox{ if } N \subset E_2 . \end{cases} \]
\end{corollary}

\begin{remark} Notice that $d_1 + d_2 = \ell$. Then a Riemann--Roch computation shows that $\Quoto_{1, \frac{1}{2} ( e - d_2 )} ( E_1 )$ and $\Quoto_{1, \frac{1}{2} ( e - d_1 )} ( E_2 )$ have expected dimension $\ell - e - (g-1)$. 
 This coincides with the expected dimension of $\IQo_{2, e} ( V )$ given in (\ref{Inle}), agreeing with Corollary \ref{rk4IQ}. \end{remark}

\begin{remark} \label{rk4OG} Note that the ruled surface $\PP E_1$ arises both as the projectivization of the isotropic subbundle $E_1$ and as a component of the isotropic Grassmannian bundle $\OG( V ) \subset \PP ( \wedge^2 V )$; that parameterizing isotropic planes of the form $(N \otimes E_2)|_p$ for $p \in C$. \end{remark}

As shown by Theorem \ref{w2deg}, if $L = \Oc$ then all rank two isotropic subbundles have degree of the same parity (which coincides with the invariant $w_2 (V)$, by Theorem \ref{w2deg}). The above results make this particularly visible for bundles of rank four:

\begin{corollary} \label{parity} Let $V = E_1 \otimes E_2$, where $\det E_1 \otimes \det E_2 \cong \Oc$. Then $E_1$ and $E_2$ have degrees of the same parity, which is also the parity of any rank two isotropic subbundle of $E_1 \otimes E_2$. 
\end{corollary}

\begin{proof} We have $\deg (E_1) \equiv \deg (E_2) \mod 2$ since $\det E_1 \otimes \det E_2 \cong \Oc$. By Proposition \ref{rk4IsotSubbs}, any rank two isotropic subbundle $F \subset V$ is isomorphic to $E_i \otimes N$ for some line subbundle $N$ and for some $i \in \{ 1, 2 \}$. Thus
\[ \deg ( F ) \ = \ \deg ( E_i \otimes N ) \ = \ \deg (E_i) + 2 \deg N \ \equiv \ \deg (E_i) \mod 2 . \qedhere \]
\end{proof}

We turn now to isotropic line subbundles. By Proposition \ref{rk4envelop}, an isotropic line subbundle $N$ of $V = E_1 \otimes E_2$ corresponds to a section of $Q_V = \PP E_1 \times_C \PP E_2$. Hence $N$ is of the form $N_1 \otimes N_2$ where $N_i$ is a line subbundle of $E_i$ for $i=1,2$. Therefore, we obtain at once:

\begin{theorem} \label{rk4rk1} Let $V = E_1 \otimes E_2$ be as above. Then $( N_1 , N_2 ) \mapsto N_1 \otimes N_2$ defines an isomorphism
\[
\bigcup_{d_1 + d_2 = d} \Quoto_{1,d_1} (E_1) \times \Quoto_{1,d_2} (E_2) \ \isom \ \IQo_{1,d} (E_1 \otimes E_2) .
\]
\end{theorem}

\subsection{Enumeration of maximal degree isotropic subbundles}

Firstly, we discuss the number of isotropic subbundles of rank two with maximal degree. In particular, we obtain another proof of Lemma \ref{zerodim} for $r = 4$.

\begin{theorem} \label{rk4enum} Let $V$ be a stable rank four orthogonal bundle which is general in its component of moduli.
\begin{enumerate}
\item[(a)] Suppose $V \in \MO(4, \Oc)_\st$. Then the number of isotropic subbundles of rank two of maximal degree is given as follows.
\[
\begin{array}{c|c|c} \hline
 & \MO^0 (4, \Oc) & \MO^1(4, \Oc) \\ \hline
 g \ \text{even} & \infty \ \text{of degree } -g & 2 \cdot 2^g  \ \text{of degree } 1-g \\ \hline
 g \ \text{odd} & 2 \cdot 2^g \ \text{of degree } 1-g  & \infty \ \text{of degree } -g \\ \hline
\end{array}
\]
\item[(b)] Suppose $V \in \MO(4, \Oc(x))_\st$. Then $V$ admits $2^g$ isotropic subbundles of rank two and degree $2-g$.

\item[(c)] Suppose $g \ge 3$ is odd. Then $V$ has $4^g$ maximal isotropic line subbundles of degree $1-g$.

\end{enumerate} \end{theorem}

\begin{proof} (a) A computation using Corollary \ref{rk4IQ} shows that for each value of $w_2 (V)$, the scheme of maximal rank two isotropic subbundles of $V = E_1 \otimes E_2$ is canonically identified with the union of the Quot schemes of maximal line subbundles of $E_1$ and $E_2$. By Theorem \ref{rk4moduli}, we may assume that $E_1$ and $E_2$ are general in moduli. The statement now follows from Theorem \ref{MaxLineSubbs}. The proof of (b) is similar. 

As for (c): It follows from Theorem \ref{rk4rk1} that the degree of an isotropic line subbundle $N_1 \otimes N_2 \subset E_1 \otimes E_2$ is maximal among the isotropic line subbundles if and only if $N_i$ is a maximal line subbundle of $E_i$ for $i = 1 , 2$. By Theorem \ref{MaxLineSubbs}, if $g$ is odd, then a general $E_i \in \SU (2, \cO_C)$ has $2^g$ maximal line subbundles of degree $-\frac{1}{2} (g-1)$. Statement (c) follows.
\end{proof}

\section{Rank three}

Suppose $V$ is an $L$-valued orthogonal bundle of rank three. By (\ref{DetOddRank}), the line bundle $L$ has even degree and $\det V \cong M^3$ for some $M$ satisfying $M^2 \cong L$.  We now analyse the structure of orthogonal bundles of rank three, expanding upon \cite[p.\ 185]{Mum} and \cite[Lemma 4.1]{CH15}.

\subsection{Rank three orthogonal bundles}

As noted in the case $L = \Oc$ in \cite[{\S} 2]{Mum}, for any rank two bundle $E$, the bundle $M \otimes (\det E)^{-1} \otimes \Sym^2 E$ has an $L$-valued quadratic form induced by
\begin{equation} ( e_1 \cdot e_2 ) \otimes ( f_1 \cdot f_2 ) \ \mapsto \ 
2 \left( e_1 \wedge f_1 \otimes e_2 \wedge f_2 + e_2 \wedge f_1 \otimes e_1 \wedge f_2 \right) . \label{rk3form} \end{equation}
We shall see that this induces finite surjective morphisms of moduli spaces. Noting that any rank three orthogonal bundle is a twist of an $\Oc$-valued orthogonal bundle of determinant $\Oc$, we now set $L = M = \Oc$ and describe the moduli space $\MO (3, \Oc)$ in more detail.

\begin{theorem} \quad \label{rk3moduli}
\begin{enumerate}
\item[(a)] The association $E \mapsto \Sym^2 E$ defines a finite surjective morphism
\[ \Theta_3^0 \colon \SU ( 2, \Oc ) \ \to \ \MO^0 ( 3, \Oc ) . \]
\item[(b)] For any $x \in C$, the association $E \mapsto \Sym^2 E \otimes \Oc (-x)$ defines a  finite surjective morphism
\[ \Theta_3^1 \colon \SU ( 2, \Oc(x) ) \ \to \ \MO^1 ( 3, \Oc ) . \]
\end{enumerate}
\end{theorem}

\begin{proof} Let us firstly define the morphism $\Theta_3^0$. As $\C$ has characteristic zero, $\Sym^2 E$ is semistable if $E$ is. Therefore, passing to a suitable \'etale cover, we can define $\Theta_3^0$ on the open subset $\SU ( 2, \Oc )_\st$ using a Poincar\'e bundle. Note that the bundle $\Sym^2 E$ and $\Sym^2 E  \otimes \Oc (-x)$ have  isotropic line subbundles of the form $N^2$ and $N^2(-x)$ respectively. Hence by Theorem \ref{w2deg}, the images of $\Theta_3^0$ and $\Theta_3^1$ land on the components $\MO^0 ( 3, \Oc )$ and $\MO^1 ( 3, \Oc )$, respectively.

Now, using for example \cite[Exercise II.5.16]{Har}, it is straightforward to check that if $E \sim E'$, then $\Sym^2 E \sim \Sym^2 E'$. Hence, as any semistable bundle is a limit of stable ones, it follows that $\Theta^3_0$ has a continuous extension to $\SU ( 2, \Oc )$. As in the proof of Theorem \ref{rk4moduli}, we can show that $\Theta_3^0$ is a morphism by using the Riemann extension theorem.

As already noted, the moduli space $\SU ( 2 , \Oc )$ has Picard number 1. Let $\cN$ be any very ample line bundle on $\MO^0 ( 3, \Oc )$. Then $\left( \Theta_3^0 \right)^* \cN \cong \cL^a$ for some positive integer $a$, where $\cL$ is as before the ample generator of $\Pic \left( \SU ( 2, \Oc ) \right)$.  Therefore, as $\Theta_3^0$ is a morphism, it must be finite. Counting dimensions of the projective varieties $\SU ( 2, \Oc ) $ and $ \MO^0 ( 3, \Oc )$,  it must be surjective. This proves (a). The proof of (b) is similar. 
\end{proof}

As in the rank four case, the following theorem ensures that an arbitrary (possibly non-stable) $L$-valued orthogonal bundle of rank three has a similar expression, and its proof shows how to find it.

\begin{theorem} \label{rk3structure} Let $V$ be a rank three $L$-valued orthogonal bundle of determinant $M^3$ as above. Then $V$ is isomorphic as an orthogonal bundle to $\Sym^2 E \otimes ( \det E)^{-1} \otimes M$ for some rank two bundle $E$. \end{theorem}

\begin{proof} We use an idea from \cite{CH15}, also applied in \cite{CCH3}. The isomorphism $M^2 \isom L$ gives $M$ the structure of an $L$-valued orthogonal bundle. The orthogonal direct sum $V \perp M$ is then $L$-valued orthogonal of rank four and determinant $M^4 = L^2$. By Proposition \ref{DetIrns}, we can find a rank two isotropic subbundle in $V \perp M$. By Theorem \ref{rk4structure}, there exist rank two bundles $E_1$, $E_2$ such that $V \perp M \cong E_1 \otimes E_2$.

Now the subbundle $M \subset E_1 \otimes E_2$ is nowhere isotropic, being an orthogonal direct summand. But as noted in Remark \ref{rk4decomp} (b), the locus of isotropic vectors corresponds precisely to the locus of maps $E_1^* \to E_2$ of rank at most one. Therefore, the image of $M \to E_1 \otimes E_2$ is indecomposable at all points. Hence it defines an isomorphism $E_1^* \otimes M \isom E_2$. Moreover, $E_1^* \otimes M \cong E_1 \otimes (\det E_1)^{-1} \otimes M$ since $E_1$ has rank two. It follows that
\[ V \perp M \ \cong \ \left( \Sym^2 E_1 \otimes ( \det E_1 )^{-1} \otimes M \right) \oplus M . \]
By uniqueness of direct sum decomposition, we conclude that $V \cong \Sym^2 E_1 \otimes ( \det E_1 )^{-1}  \otimes M$. By \cite[Proposition 3.1]{Gro}, composing with an automorphism of $V$ if necessary, this may be assumed to be an isomorphism of orthogonal bundles. \end{proof}

\subsection{Quadric fibration and isotropic subbundles}

We consider an $L$-valued rank three orthogonal bundle $V = \Sym^2 E \otimes ( \det E )^{-1} \otimes M$ of determinant $M^3$ as above. We now describe the locus of isotropic lines $Q_V \subset \PP V$.

\begin{proposition} \label{rk3envelop} The fibration $Q_V$ coincides with the relative Segre embedding
\[ \PP E \ \hookrightarrow \ \PP \Sym^2 E \ \isom \ \PP ( \Sym^2 E \otimes ( \det E )^{-1} \otimes M ) . \]
\end{proposition}

\begin{proof} This follows easily from (\ref{rk3form}). 
 For a geometric argument: By Theorem \ref{rk3structure}, we may assume that the quadratic form is induced from that $V \perp M \cong E \otimes E \otimes (\det E)^{-1} \otimes M$. Thus $Q_V = Q_{V \perp M} \cap \PP V$. By Proposition \ref{rk4envelop}, this is
\[ \left( \PP E \times_C \PP E \right) \cap \left( \PP ( \Sym^2 E \otimes (\det E)^{-1} \otimes M ) \right) , \]
which is precisely the image of the relative Segre embedding, as desired. \end{proof}

\begin{theorem}[Isotropic Quot schemes in rank three] \label{rk3IQ} Let $V = \Sym^2 E \otimes (\det E)^{-1} \otimes M$ be as above. Write $d := \deg(E)$ and $m := \deg (M)$.
\begin{enumerate}
\item[(a)] For each integer $e$, there is an isomorphism
\[ \Quoto_{1, \frac{1}{2}(e+d-m)} ( E ) \ \isom \ \IQo_{1, e} ( V ) \]
given by $N_1 \mapsto   N_1^2 \otimes (\det E)^{-1} \otimes M$.
\item[(b)] In particular, if $N \subset V$ is any isotropic line subbundle, then $\deg (N) \equiv d - m \mod 2$.
\end{enumerate}
\end{theorem}

\begin{proof} 
 Let $N \subset V$ be an isotropic subbundle of degree $e$. Then $N \otimes (\det E) \otimes M^{-1}$ is a line subbundle of $\Sym^2 E$. By Proposition \ref{rk3envelop}, this belongs to the cone over the relative Segre embedding $\PP E \hookrightarrow \PP \Sym^2 E$. Thus it is of the form $\Sym^2 N_1$ for some $N_1 \subset E$ of degree $\frac{1}{2}(e + d - m)$. 
%
 In particular, $e \equiv d - m \equiv 0 \mod 2$. Clearly $N_1$ is uniquely determined as an element of $\Quoto_{1, \frac{1}{2}(e + d - m)} (E)$. The statement follows. \end{proof}

\begin{remark} A Riemann--Roch computation shows that $\Quoto_{1, \frac{1}{2} ( e + d - m )} ( E )$ has expected dimension $m - e - (g-1)$, which coincides with the expected dimension of $\IQo_{1, e} ( V )$ given in (\ref{Inle}), 
 in agreement with Corollary \ref{rk4IQ}. \end{remark}

\subsection{Enumeration of isotropic line subbundles}

Together with Theorem \ref{MaxLineSubbs}, the above results allow us to count the number of isotropic line subbundles of maximal degree of a general stable $V \in \MO^i (3, \Oc)$, when this is finite.


\begin{theorem} \label{rk3enum} Let $V$ be a stable $\Oc$-valued orthogonal bundle of rank three and trivial determinant.
\begin{enumerate}
\item[(a)] If $g$ is even and $V$ is general in $\MO^1 (3, \Oc)_\st$, then $V$ has $2^g$ maximal isotropic line subbundles of degree $1-g$.
\item[(b)] If $g$ is odd, and $V$ is general in $\MO^0 (3, \Oc)_\st$, then $V$ has $2^g$ maximal isotropic line subbundles of degree $1-g$.
\end{enumerate}
\end{theorem}

\begin{proof} By Theorem \ref{rk3IQ}, maximal line subbundles of $E$ are in bijection with maximal isotropic line subbundles of $\Sym^2 E \otimes (\det E)^{-1}$. If $g$ is even, then by Theorem \ref{MaxLineSubbs}, a general $E \in \SU (2, \Oc(x) )_\st$ has $2^g$ line subbundles $N_1$ of degree $1-\frac{g}{2}$. By Theorem \ref{rk3IQ}, these induce $2^g$ isotropic line subbundles $N_1^2 (-x) \subset \Sym^2 E \otimes \Oc(-x)$ of degree $1-g$, and we obtain (a). The proof of (b) is similar. 
\end{proof}

%
%

\section{Rank six} \label{rk6}

As was the case for orthogonal bundles of ranks three and four, the rank five case turns out to be closely related to the rank six case. Therefore we now turn to $L$-valued orthogonal bundles of rank six admitting rank three isotropic subbundles; equivalently, in view of Lemma \ref{DetIrns}, having determinant $L^3$.

\subsection{Rank six orthogonal bundles with rank three isotropic subbundles}

If $W$ is any bundle of rank four and determinant $L$, then there is a natural map
\begin{equation} \label{rk6form} \wedge^2 W \otimes \wedge^2 W \ \to \ \wedge^4 W \ = \ L \end{equation}
given by the determinant. This endows $\wedge^2 W$ with the structure of an $L$-valued orthogonal bundle. Furthermore, every rank three subbundle $F \subset W$ defines a rank three subbundle $\wedge^2 F \subset \wedge^2 W$, which is isotropic since $\wedge^4 F = 0$. We shall now show that up to a twist, any rank six orthogonal bundle admitting a rank three isotropic subbundle is of the form $\wedge^2 W$. (In {\S} \ref{rk6QV}, however, we shall see that not every rank three isotropic subbundle need be of the form $\wedge^2 F$, even up to twisting.)

As in the previous sections, we shall use this to give a description of the moduli spaces of rank six orthogonal bundles.
   
\begin{theorem} \label{rk6moduli} Let $x$ be any point of $C$.
\begin{enumerate}
\item[(a)] The association $W \mapsto \wedge^2 W$ defines a finite   surjective morphism
\[ \Theta_6^0 \colon \SU ( 4, \Oc ) \ \to \ \MO^0 ( 6, \Oc ) . \]
\item[(b)] The association $W_1 \mapsto \wedge^2 W_1 \otimes \Oc (x)$ defines a finite surjective morphism
\[ \Theta_6^1 \colon \SU ( 4, \Oc (-2x) ) \to \MO^1 ( 6, \Oc ) . \]
\item[(c)] The association $W_2 \mapsto \wedge^2 W_2$ defines a finite  surjective morphism
\[ \Phi_6 \colon \SU ( 4, \Oc (x) ) \ \to \ \MO ( 6, \Oc (x) ) . \]
\end{enumerate}
\end{theorem}

\begin{proof} This is similar to the rank three case. As $\C$ has characteristic zero, $\wedge^2 W$ is semistable whenever $W$ is. Using for example \cite[Exercise II.5.16]{Har}, one checks that if $W$ and $W'$ are S-equivalent semistable bundles of rank four, then  $\wedge^2 W$ and $\wedge^2 W'$ are S-equivalent. By a similar argument to that in Theorem \ref{rk3moduli}, we obtain the existence, finiteness and surjectivity of the morphisms $\Theta^i_6$ and $\Phi_6$. The target component of each $\Theta_6^i$ is determined by  Theorem \ref{w2deg} and the parity of
\[
\deg (\wedge^2 F) = 2 \cdot \deg (F)  \ \text{ and  } \ \deg (\wedge^2 F)  \otimes \Oc (x) = 2 \cdot \deg (F) + 3 
\]
 for rank three subbundles $F \subset W$.
\end{proof}

 As in the previous cases, we describe a construction required for the next proposition which shows how to get the ``inverses'' of the morphisms $\Theta_6^i$ and $\Phi_6$. Let $V$ be a rank six $L$-valued orthogonal bundle admitting a rank three isotropic subbundle $E$. Then by Criterion \ref{extension}, there is an exact sequence $0 \to E \to V \to E^* \otimes L \to 0$, 
 defining an extension class $\delta \in H^1 ( C, \wedge^2 E \otimes L^{-1} )$. Then $\delta$ also defines an extension $0 \to \wedge^2 E \to W \to L \to 0$ of rank four.

\begin{proposition} \label{rk6StructureProp} Let $V$ be an $L$-valued orthogonal bundle of rank six. Let $E$, $\delta$ and $W$ be as above. Then $V$ is isomorphic to $( \wedge^2 W ) \otimes  (\det E)^{-1}$. \end{proposition}

\begin{proof} To ease notation, write $\hE := \wedge^2 E$, a subbundle of $W$. We consider the antisymmetrization map $a \colon W \otimes W \to \wedge^2 W$. By linear algebra, using the fact that $\hE$ has corank one, the restriction of $a$ to $W \otimes \hE$ is surjective with kernel $\Sym^2 \hE$. Therefore, we have a diagram
\[ \xymatrix{ & \Sym^2 \hE \ar[r]^= \ar[d] & \Sym^2 \hE \ar[d] & & \\
 0 \ar[r] & \hE \otimes \hE \ar[r] \ar[d]^a & W \otimes \hE \ar[r] \ar[d]^a & L \otimes \hE \ar[r] \ar[d] & 0 \\
 0 \ar[r] & \wedge^2 \hE \ar[r] & \wedge^2 W \ar[r] & L \otimes \hE \ar[r] & 0. } \]
In particular, $\wedge^2 W$ is an extension of class
\begin{equation} a_* ( \delta \otimes \Iden_\hE ) \ \in \ H^1 ( C, \Hom ( L \otimes \hE , \wedge^2 \hE ) ) . \label{WExtClass} \end{equation}

Now write $M := \wedge^3 E = \det E$. The association
\[ u \wedge v \ \mapsto \ \left( w \ \mapsto u \wedge v \wedge w \right) \]
gives a canonical identification
\begin{equation} \hE \ = \ \wedge^2 E \ = \ \Hom ( E , \det E ) \ \isom \ E^* \otimes M . \label{wedge2identif} \end{equation}
Similarly, we have canonical identifications
\begin{equation} \wedge^2 \hE \ \isom  \ \wedge^2 ( E^* \otimes M ) \ \cong \ \wedge^2 E^* \otimes M^2 \ \isom \ E \otimes M . \label{6two} \end{equation}
Thus $\wedge^2 W$ is an extension $0 \to E \otimes M \to \wedge^2 W \to E^* \otimes M \otimes L \to 0$, and
\[ \Hom ( L \otimes \hE , \wedge^2 \hE ) \ = \ L^{-1} \otimes E \otimes M^{-1} \otimes E \otimes M \ \cong \ L^{-1} \otimes E \otimes E \otimes \End \; M . \]
Therefore, to prove the theorem, in view of (\ref{WExtClass}), (\ref{wedge2identif}) and (\ref{6two}) it will suffice to prove the following lemma. \end{proof}

\begin{lemma} The element $a_* ( p \otimes \Iden_\hE )$ corresponds to $-p \otimes \Iden_M$ for each $p \in L^{-1} \otimes \wedge^2 E$. \end{lemma}

\begin{proof} In what follows, we shall assume that $L = \Oc$, as the general case is only notationally more cumbersome. Let $p$ be any element of $L^{-1} \otimes \wedge^2 E|_x = \wedge^2 E|_x$ for some $x \in C$. As $E$ has rank three, we can write $p = u \wedge v$ for some $u, v \in E|_x$. If $p$ is nonzero then we can choose $w \in E|_x$ such that $u , v , w$ is a basis of $E|_x$. Let $u^* , v^* , w^*$ be the dual basis of $E^*|_x$. Set $m := u \wedge v \wedge w \in M|_x$. Then $u^* \wedge v^* \wedge w^* = m^*$ where $m^* \in M^*|_x$ is dual to $m$ (that is, $\langle m^* , m \rangle = 1$). Then via (\ref{wedge2identif}), we have identifications
\begin{equation} u \wedge v \ \leftrightarrow \ w^* \otimes m , \quad u \wedge w \ \leftrightarrow \ (-v)^* \otimes m , \quad \hbox{and} \quad v \wedge w \ \leftrightarrow \ u^* \otimes m . \label{DualBases} \end{equation}
Via these identifications, and using the contraction $m^* \otimes m \mapsto 1$, the map $\Iden_{\wedge^2 E}|_x$ can be identified with
\[ ( u \wedge v ) \otimes m^* \otimes w + ( u \wedge w ) \otimes m^* \otimes ( -v ) + ( v \wedge w ) \otimes m^* \otimes u . \]
Thus $p \otimes \Iden_{\wedge^2 E}$ is expressed as
\[ ( u \wedge v ) \otimes ( u \wedge v ) \otimes m^* \otimes w + ( u \wedge v ) \otimes ( u \wedge w ) \otimes m^* \otimes ( -v ) + ( u \wedge v ) \otimes ( v \wedge w ) \otimes m^* \otimes u . \]
Substituting from (\ref{DualBases}) (thereby using the first identification in (\ref{6two})), we obtain
\[ a_*(p \otimes \Iden_{\wedge^2 E})  =   ( w^* \otimes m ) \wedge ( (-v^*) \otimes m ) \otimes m^* \otimes ( -v ) + ( w^* \otimes m ) \wedge ( u^* \otimes m ) \otimes m^* \otimes u . \]
Applying the contraction $m \otimes m^* \mapsto 1$, this becomes
\begin{equation} ( w^* \wedge v^* ) \otimes v \otimes m + ( w^* \wedge u^* ) \otimes u \otimes m . \label{6one} \end{equation}
Now since $u^* \wedge v^* \wedge w^* = m^*$, in particular $w^* \wedge v^*$ is the dual basis element to $(-u)^* \otimes m$. Thus via the identification $\wedge^2 E^* \isom \Hom ( E^* , M^{-1} ) \cong E \otimes M^{-1}$, we see that $w^* \wedge v^*$ corresponds to $(-u) \otimes m^*$. Similarly, $w^* \wedge u^*$ corresponds to $v \otimes m^*$. Thus (\ref{6one}) is identified with
\[ ( -u ) \otimes m^* \otimes v \otimes m + v \otimes m^* \otimes u \otimes m , \]
corresponding to $-(u \wedge v) \otimes \Iden_M = - p \otimes \Iden_M$, as desired. \end{proof}

We can now give a structure theorem for orthogonal bundles of rank six. 

\begin{theorem} \label{rk6Structure} Let $V$ be an $L$-valued orthogonal bundle of rank six.
\begin{enumerate}
\item[(a)] The bundle $V$ is isomorphic to $\wedge^2 W$ for some $W$ of rank four and determinant $L$ if and only if $V$ admits a rank three isotropic subbundle of even degree.
\item[(b)] Let $x$ be any point of $C$. Then $V$ is isomorphic to $\wedge^2 W_1 \otimes \Oc(x)$ for some $W_1$ of rank four and determinant $L(-2x)$ if and only if $V$ admits a rank three isotropic subbundle of odd degree.
\end{enumerate}
Note that (a) and (b) apply simultaneously when $\deg (L)$ is odd.
\end{theorem}

\begin{proof} As the proofs of (a) and (b) are almost identical, we prove (b) only. Let $E$ be a rank three subbundle of $V$ which is isotropic and of odd degree. By Proposition \ref{rk6StructureProp}, we have $V \cong \wedge^2 W \otimes \det E^*$ where $0 \to \wedge^2 E \to W \to L \to 0$ is constructed from $V$ and $E$ as previously. By hypothesis, we can find $N$ such that $\det E \cong N^2 (-x)$. Set $W_1 := W \otimes N^{-1}$. Then
\[ \det W_1 \ \cong \ (\det E)^2 \otimes L \otimes N^{-4} \ \cong \ N^4 ( -2x ) \otimes L \otimes N^{-4} \ \cong \ L ( -2x ) , \]
and $V \cong \wedge^2 W \otimes N^{-2} \otimes \Oc (x) = ( \wedge^2 W_1 ) \otimes \Oc (x)$.

Conversely, if $V = ( \wedge^2 W_1 ) \otimes \Oc (x)$ then any rank three subbundle $F \subset W_1$ yields a subbundle $\wedge^2 F \otimes \Oc(x) \subset V$ which is isotropic since $\wedge^4 F = 0$, and has odd degree $2 \cdot \deg F + 3$. \end{proof}

\subsection{Enveloping bundle structure and isotropic subbundles} \label{rk6QV}

Let $W$ be a bundle of rank four and determinant $L$, and let $V := \wedge^2 W$ be the associated $L$-valued orthogonal bundle of rank six. By \cite[pp.\ 209--211]{GH}, an element $\omega \in \wedge^2 W$ belongs to the cone over the Pl\"ucker image of $\Gr ( 2, W )$ in $\PP ( \wedge^2 W )$ if and only if $\omega$ is decomposable; equivalently $\omega \wedge \omega = 0$ in $\wedge^4 W$. In view of (\ref{rk6form}), this is equivalent to $\omega$ being isotropic with respect to the orthogonal structure on $\wedge^2 W$. It follows that:

\begin{proposition} \label{rk6Envelop} The relative Pl\"ucker map embeds $\Gr ( 2, W )$ in $\PP ( \wedge^2 W )$ as the projectivization $Q_{\wedge^2 W}$ of the locus of vectors isotropic with respect to the orthogonal structure on $\wedge^2 W$. In particular, $\PP ( \wedge^2 W )$ is an enveloping bundle for $Q_{\wedge^2 W} = \Gr ( 2, W ) \to C$. \end{proposition}

We shall also require the following lemma on isotropic subbundles of orthogonal bundles.

\begin{lemma} \label{SimpleIsotropyCrit} Let $V$ be an $L$-valued orthogonal bundle (of arbitrary rank) and $E \subset V$ any subbundle. Then $E$ is isotropic if and only if $\PP E \subset Q_V$. \end{lemma}

\begin{proof} One direction is clear from the definition of $Q_V$. Conversely, suppose $\PP E \subset Q_V$. Let $v, w$ be elements of a fiber $E|_x$. By symmetry of $\sigma$, we have
\[ 2 \cdot \sigma ( v \otimes w ) \ = \ \sigma \left( (v + w) \otimes ( v + w ) \right) - \sigma ( v \otimes v ) - \sigma ( w \otimes w ) . \]
By hypothesis, the right hand side is zero, so $\sigma ( v \otimes w ) = 0$. \end{proof}

\begin{remark} Note that the above proof depends on the symmetry of the bilinear form, and may fail in characteristic $2$. \end{remark}

We shall use these observations to classify isotropic subbundles of $V$, beginning with rank three. By Proposition \ref{rk6Envelop} and Lemma \ref{SimpleIsotropyCrit}, a rank three subbundle $E \subset V$ is isotropic if and only if the $\PP^2$-subbundle $\PP E \subset \PP V$ is contained in $\Gr ( 2, W )$. For such an $E$, by \cite[p.\ 757]{GH}, one of the following two situations must arise:
\begin{enumerate}
\item[(i)] For each $p \in C$, there exists a three-dimensional subspace $F_p \subset V|_p$ such that
\[ \PP E|_p \ = \ \{ \Lambda \in \Gr ( 2, W|_p ) : \Lambda \subset F_p \} . \]
\item[(ii)] For each $p \in C$, there exists a line $N_p \subset W|_p$ such that
\[ \PP E|_p \ = \ \{ \Lambda \in \Gr ( 2, W|_p ) : N_p \subset \Lambda \} . \]
\end{enumerate}

If situation (i) arises, then, using local triviality, there exists a rank three subbundle $F \subset W$ such that $\PP E = \Gr ( 2, F )$. By definition of the Pl\"ucker embedding $\Gr ( 2, F ) \to \PP ( \wedge^2 F ) \subset \PP ( \wedge^2 W )$ it follows that $\PP E$ is the projectivization of the set of bivectors
\[ \{ v \wedge w : v, w \in F \} . \]
As $F$ has rank three, every element of $\wedge^2 F$ is of this form. Thus $E = \wedge^2 F$. In particular, $\deg E = 2 \cdot \deg F$ is even.

On the other hand, suppose situation (ii) arises. Then there is a line subbundle $N \subset W$ such that $\PP E$ is the projectivization of $\{ v \wedge w : v \in N, w \in W \}$. We define a map $\tau_N \colon N \otimes W \to \wedge^2 W$ by
\begin{equation}   \tau_N ( v \otimes w ) \ = \ v \wedge w . \label{TauN} \end{equation} 
Clearly $\tau_N$ has kernel $N \otimes N$ and image exactly $E$. Thus $E \cong N \otimes \frac{W}{N}$, and
\[ \deg E \ = \ 3 \cdot \deg N + \deg W - \deg N \ = \ 2 \cdot \deg N + \deg L . \]
In particular, $\deg E \equiv \deg L \mod 2$. 
 Now we can summarize the above discussion as follows.

\begin{theorem} \label{rk6IsotSubbs} Let $V = \wedge^2 W$ be a rank six orthogonal bundle for a rank four bundle $W$ of determinant $L$. Set $\ell := \deg L$ and let $d$ be any integer.
\begin{enumerate}
\item[(a)] Suppose $\ell$ is even. Then there is an isomorphism
\[ \Quoto_{3, d} ( W ) \ \sqcup \ \Quoto_{1, d - \frac{\ell}{2}} ( W ) \ \isom \ \IQo_{3, 2d} (  V) . \]
In particular, $\IQo_{3, 2d} ( V )$ is disconnected for $d \ll 0$.
\item[(b)] Suppose $\ell$ is odd. Then there are isomorphisms
\[ \Quoto_{3, d} ( W ) \ \isom \ \IQo_{3, 2d} (  V ) \quad \hbox{and} \quad \Quoto_{1, d - \frac{\ell+1}{2}} ( W ) \ \isom \ \IQo_{3, 2d-1} (  V ) . \]
\end{enumerate}
In both cases, the isomorphisms are given respectively by
\[ E \ \mapsto \ \wedge^2 E \quad \hbox{and} \quad N \ \mapsto \ \tau_N ( N \otimes W ) \ \cong \ N \otimes \frac{W}{N} . \]
\end{theorem}

\begin{remark} A Riemann--Roch computation shows that $\Quot_{3, d} (W)$ and $\Quot_{1, d - \frac{\ell + \varepsilon}{2}} ( W )$ have the same expected dimensions respectively as those given in (\ref{Inle}) for $\IQo_{3, 2d} ( \wedge^2 W )$ and $\IQo_{3, 2d - \varepsilon} ( \wedge^2 W )$, in agreement with Theorem \ref{rk6IsotSubbs}. \end{remark}

Lastly, we treat the case where all rank three isotropic subbundles have odd degree, so $V$ is of the form $\wedge^2 W \otimes \Oc(x)$.

\begin{corollary} \label{rk6IsotSubbsCor} Suppose $\ell$ is even. Let $W$ be a rank four bundle of determinant $L$, and let $V$ be the orthogonal bundle $\wedge^2 W \otimes \cO_C(x)$. Then for any $d$, there are isomorphisms
\[ \Quoto_{3, d-1} ( W ) \ \sqcup \ \Quoto_{1, d - 1 - \frac{\ell}{2}} ( W ) \ \isom \ \IQo_{3, 2d + 1} ( V) , \]
given by
\[ \begin{cases} E \ \mapsto \ \wedge^2 E \otimes \Oc(x) \hbox{ if } E \in \Quoto_{3, d-1} ( W ) ; \\
 M \ \mapsto \ \tau_N ( N \otimes W ) \otimes \Oc (x) \ \cong \ N \otimes \frac{W}{N} \otimes \Oc (x) \hbox{ if } E \in \Quoto_{1, d - 1 - \frac{\ell}{2}} ( W ) . \end{cases} \]
In particular, $\IQo_{3, 2d + 1} (V)$ is disconnected for $d \ll 0$.
\end{corollary}

\begin{proof} For any rank six $L$-valued orthogonal $V$, tensor product by $\Oc(x)$ defines a canonical isomorphism $\IQo_{3, 2d-2} (V) \isom \IQo_{3, 2d + 1} ( V \otimes \Oc (x) )$. Thus the statement follows from Theorem \ref{rk6IsotSubbs} (a). \end{proof}

\begin{remark} The fact that the isotropic Quot schemes of $W$ may have two connected components reflects that fact that the isotropic Grassmann bundle $\OG ( 3, \wedge^2 W )$ has two components. This is studied in more generality in \cite{CCH3}. Further detail for rank six is given in Remark \ref{LinkIsotSubbsFiveSix}. \end{remark}


Let us now consider isotropic subbundles of rank two or one. For brevity, we shall suppose throughout that $V$ is of the form $\wedge^2 W$; the case $V = \wedge^2 W \otimes \Oc (x)$ is left to the reader.

Let $F \subset V$ be a rank two isotropic subbundle. This corresponds to a $\PP^1$-subbundle of the Grassmannian bundle $\Gr(2, W) \to C$. By \cite[pp.\ 756--757]{GH}, any $\PP^1$ lying inside $\Gr(2, 4)$ is of the form 
\[
\{ \Lambda \in \Gr (2, 4) \: : \: p \subset \PP \Lambda \subset h \}
\]
for some fixed point $p$ and hyperplane $h$ in $\PP^3$. Using local triviality as above, there exists a flag of subbundles $N \subset H \subset W$ with $\rank ( N ) = 1$ and $\rank (H) = 3$ such that for each $x \in C$ we have
\[
\PP F|_x \: = \: \{ \Lambda \in \Gr(2, W|_x) \: : \: N|_x \subset \Lambda \subset H|_x \}. 
\]
It follows that 
\[
F \: = \: \{ v \wedge w \: : \: v \in N, \ w \in H \} \ = \ \tau_N (N \otimes H), 
\]
where $\tau_N$ is the map defined in (\ref{TauN}).  Hence $F \cong (N \otimes H) / (N \otimes N)$ and $\deg (F) = \deg (H) + \deg (N)$. As a consequence, we obtain:

\begin{theorem} \label{rk6rk2} For any rank four bundle $W$, there is an identification 
\[
\IQo_{2,d} ( \wedge^2 W ) \cong \bigcup_{d_1+d_2=d} \cQ uot^\circ_{1,d_1} (\cH) ,
\]
where $\pi \colon \cQ uot_{1,d_1}^\circ (\cH) \to \Quoto_{3,d_2} (W)$ is the relative Quot scheme with $\pi^{-1}(H) = \Quot_{1,d_1}^\circ(H)$ at $H \in \Quoto_{3,d_2} (W)$, which parameterizes the flags $N \subset H \subset W$ given as above. \end{theorem}

\begin{remark} Notice that the rank two isotropic subbundle $\tau_N ( N \otimes H )$ is exactly the intersection of the rank three isotropic subbundles $\tau_N ( N \otimes W )$ and $\wedge^2 H$, which belong to opposite components of $\OG ( V )$. Compare with \cite[Proposition, p.\ 735]{GH}. We discuss another aspect of this in Remark \ref{LinkIsotSubbsFiveSix}. \end{remark}

Lastly: By Theorem \ref{rk6Envelop}, an isotropic line subbundle of $\wedge^2 W$ is simply a section of $\Gr(2, W) \to C$. But this is equivalent to a choice of rank two subbundle $E \subset W$, and then $N = \wedge^2 E$. Thus we have

\begin{theorem} \label{rk6rk1} Let $W$ be any rank four bundle and $\wedge^2 W$ the associated orthogonal bundle. Then $\IQo_{1,d} (\wedge^2 W) \cong \Quoto_{2,d}(W)$. \end{theorem}

\subsection{Enumeration of maximal isotropic subbundles} 

We now apply the previous results to the problem of enumerating isotropic subbundles of maximal degree of a general rank six $L$-valued orthogonal bundle $V$, when there are finitely many of these.  For ranks one and two, for brevity we restrict to the case $V \cong \wedge^2 W$ with $\det ( W ) \cong \Oc$.

\begin{theorem} \label{rk6enum} Let $V$ be a rank six $\Oc$-valued orthogonal bundle which is general in its component of moduli.
\begin{enumerate}
\item[(a)] Suppose $g \equiv 1 \mod 4$ and $w_2 ( V ) = 0$. Then $V$ has $2 \cdot 4^g$ rank three isotropic subbundles of maximal degree $-\frac{3}{2} ( g-1 )$.
\item[(b)] Suppose $g \equiv 3 \mod 4$ and $w_2 ( V ) = 1$. Then $V$ has $2 \cdot 4^g$ rank three isotropic subbundles of maximal degree $-\frac{3}{2} ( g-1 )$.
\item[(c)] Suppose that $g \equiv 1 \mod 12$ and $w_2 (V) = 0$. Then $V$ has $12^g$ rank two isotropic subbundles of maximal degree $-\frac{5}{3}(g-1)$.
\item[(d)] Suppose $w_2 (V) = 0$. Then $V$ has
\[
 2^{3g-1}   - (-1)^{g} \cdot 2^{2g-1} 
\]
 isotropic line subbundles of maximal degree $1-g$.
\end{enumerate}
\end{theorem}

\begin{proof} (a) By Theorem \ref{rk6moduli} (a), we have $V \cong \wedge^2 W$ for some $W \in \SU ( 4, \Oc )$, which may assumed to be general in moduli. Note that $W^*$ may also be assumed to be general in moduli, and that $\Quot_{3, d} ( W ) \cong \Quot_{1, d} ( W^* )$. Thus, as $\deg (W) \equiv 1 - g \mod 4$, by Theorem \ref{MaxLineSubbs} the bundles $W$ and $W^*$ both have $4^g$ maximal line subbundles of degree $\frac{3}{4}(g-1)$. By Theorem \ref{rk6IsotSubbs} (a), the bundle $V$ contains $2 \cdot 4^g$ rank three isotropic subbundles of maximal degree $-\frac{3}{2} (g-1)$.

(b) By Theorem \ref{rk6moduli} (b), we have $V \cong \wedge^2 W_1 \otimes \Oc (x)$ for some $W_1 \in \SU ( 4, \Oc ( -2x ))$, which may be assumed to be general in moduli. As $-2 \equiv 1-g \mod 4$, by Theorem \ref{MaxLineSubbs} both the bundles $W_1$ and $W_1^*$ have $4^g$ maximal line subbundles of degree $d$ and $d+1$ respectively, for $d = \frac{1}{4} ( 1 - 3g)$. By Corollary \ref{rk6IsotSubbsCor}, we have
\[ \IQo_{3, 2d+1} ( V ) \ \cong \ \Quoto_{3, d-1} ( W_1 ) \sqcup \Quoto_{1, d} ( W_1 ) . \]
Since $\Quoto_{3, d-1} ( W_1 ) \cong \Quoto_{1, d + 1} ( W_1^* )$, we get the desired number.

(c) Here $V = \wedge^2 W$ is as in part (a). By hypothesis, in particular $g \equiv 1 \mod 4$, so a general $W \in \SU (4, \cO_C)$ has $4^g$ maximal rank three subbundles of degree $-\frac{3}{4}(g-1)$. By \cite[Lemma 2.1]{LNew}, these maximal subbundles are general in moduli. As by hypothesis we also have $g \equiv 1 \mod 3$, each maximal rank three $H \subset W$ has $3^g$ maximal line subbundles, each of degree $-\frac{11}{12} (g-1)$. Thus there are $4^g \cdot 3^g = 12^g$ flags $N \subset H \subset W$ where both $N \subset H$ and $H \subset W$ are maximal subbundles. By Theorem \ref{rk6rk2}, then, $V$ has $12^g$ isotropic subbundles of rank two and maximal degree 
\[
- \frac{3}{4} ( g - 1 ) - \frac{11}{12} ( g - 1 ) \ = \ - \frac{5}{3} ( g - 1 ) .
\]

(d) Using the formulas in \cite[{\S} 4]{Hol}, one finds that a general bundle of rank four and degree zero has a finite number of rank two subbundles of maximal degree $1-g$, the number being given as follows:
\begin{equation}  2^{4g - 3} \cdot \sum_{z \in \{ -1, \pm \sqrt{-1} \}} \left( z ( 1 - z )^2 \right)^{1-g} \ = \  2^{3g-1}   + (-1)^{g-1} \cdot 2^{2g-1}  . \end{equation} 

The statement now follows from Theorem \ref{rk6rk1}.
\end{proof}

\begin{remark} By the Hirschowitz bound \cite{Hir}, every rank six vector bundle of degree zero has a rank two subbundle of degree at least $-\frac{4}{3}(g-1)$. This means that a general $V \in \MO^0 (6, \cO_C)$ does have rank two subbundles of degree greater than $-\frac{5}{3}(g-1)$, all of which are nonisotropic. \end{remark}

\begin{theorem} Let $V$ be a rank six $\Oc(x)$-valued orthogonal bundle of determinant $\Oc (3x)$ which is general in moduli. Suppose $g$ is even. Then $V$ has $4^g$ isotropic rank three subbundles of maximal degree $3 - \frac{3}{2}g$. \end{theorem}

\begin{proof} By Theorem \ref{rk6moduli} (c), we have $V \cong \wedge^2 W_2$ for some $W_2 \in \SU ( 4, \Oc (x))$, which may be assumed to be general in moduli. There are two cases.

If $g \equiv 0 \mod 4$, put  $ d := 2 - \frac{3}{4} g$, so that $ 2d - 1 = 3 - \frac{3}{2}g  $. Then by Theorem \ref{rk6IsotSubbs} (b), we have $\IQo_{3, 2d-1} (V) \cong \Quoto_{1, d-1} (W_2)$.

If $g \equiv 2 \mod 4$, put $d := \frac{3}{2} - \frac{3}{4}g$, so that $2d = 3 - \frac{3}{2}g$. Then by Theorem \ref{rk6IsotSubbs} (b), we have $\IQo_{3, 2d} ( V ) \cong \Quoto_{3 , d} (W_2) \cong \Quoto_{1, d-1} ( W_2^* )$.

Applying Theorem \ref{MaxLineSubbs} to $\Quoto_{1, d-1} ( W_2 )$ and $\Quoto_{1 , d-1} (W_2^*)$, we get the desired number. \end{proof}

\section{Rank five}

Let $L$ be a line bundle, and $V$ an $L$-valued orthogonal bundle of rank five. By (\ref{DetOddRank}), necessarily $L$ has even degree $2m$, and $\det W \cong M^5$ for some $M$ satisfying $M^2 \cong L$. We shall use the results in {\S} \ref{rk6} to analyze the structure of $V$, in particular the relation with rank four symplectic bundles.

\subsection{Structure of rank five orthogonal bundles}

Let $M$ be a line bundle and $W$ a rank four bundle with a symplectic form $\alpha \colon \wedge^2 W \to M$. Then $\deg ( W ) = 2 \cdot \deg (M)$, and $S_W := \Ker \left( \alpha \colon \wedge^2 W \to M \right)$ is a rank five subbundle of determinant $M^5$. Note that $S_W$ depends on the choice of $\alpha$, but if for example $W$ is stable then $\alpha$ is unique up to nonzero scalar multiple.

\begin{proposition} \label{wedge2W} The bundle $S_W$ is $M^2$-valued orthogonal, and there is an orthogonal direct sum decomposition $\wedge^2 W = S_W \perp M$. \end{proposition}

\begin{proof} As $W$ is $M$-valued symplectic, $\det W \cong M^2$. Thus $S_W$ inherits the symmetric bilinear form $\sigma \colon ( \wedge^2 W )^{\otimes 2} \to M^2$ given by the determinant, as discussed in the previous section. Let us show that $\sigma|_{\left( S_W \right)^{\otimes 2}}$ is nondegenerate. Suppose $\eta \in S_W|_p$ is such that $\sigma ( \eta \wedge - )$ is the zero functional on $S_W|_p$. Then in particular $\eta \wedge \eta = 0$, so $\eta$ is decomposable. Furthermore, as
\[ \eta \ \in \ \Ker \left( \alpha|_p \colon \wedge^2 W|_p \ \to \ M|_p \right) , \]
if $\eta$ is nonzero then $\eta = v \wedge w$ where $v$ and $w$ span a subspace $\Lambda \subset W|_p$ isotropic with respect to $\alpha$. But then for any $v', w' \in W|_p$ spanning an $\alpha$-isotropic subspace complementary to $\Lambda$, we have $v' \wedge w' \in S_W|_p$ and $v \wedge w \wedge v' \wedge w' \neq 0$ in $M^2|_p$. Thus $\eta$ must be zero, and $\sigma|_{\left( S_W \right)^{\otimes 2}}$ is nondegenerate. Hence $S_W$ is $M^2$-valued orthogonal.

For the rest: By the nondegeneracy, $S_W \cap ( S_W )^\perp = 0$ in $\wedge^2 W$. Comparing determinants, we see that $(S_W)^\perp \cong M$. Therefore, $\wedge^2 W = S_W \perp M$ as desired. \end{proof}

Again, we get morphisms of moduli spaces from this observation. For any line bundle $M$, we denote by $\MS ( 4 , M )$ the moduli space of semistable $M$-valued symplectic bundles of rank $4$ over $C$. This is a closed irreducible subvariety of $\SU ( 2n , M^n )$ of dimension $10(g-1)$.

\begin{theorem} \label{rk5moduli} Let $x$ be any point of $C$.
\begin{enumerate}
\item[(a)] The association $W \mapsto S_W$ defines a finite surjective morphism
\[ \Theta_5^0 \colon \MS ( 4, \Oc ) \ \to \ \MO^0 ( 5, \Oc ) . \]
\item[(b)] The association $W_1 \mapsto S_{W_1} \otimes \Oc (x)$ defines a finite surjective morphism
\[ \Theta_5^1 \colon \MS ( 4, \Oc (-x) ) \to \ \MO^1 ( 5, \Oc ) . \]
\end{enumerate}
\end{theorem}

\begin{proof} Let $\tM$ be an \'etale cover of $\MS ( 4, \Oc )_\st$ admitting a Poincar\'e family $\cW \to \tM \times C$ together with a symplectic form $\tilde{\alpha} \colon \wedge^2 \cW \to \cO_{\tM \times C}$. Then  by Proposition \ref{wedge2W} we see that $S_{\cW} := \Ker ( \tilde{\alpha} )$ is a family of rank five orthogonal bundles. Moreover, since $\C$ has characteristic zero, if a bundle $W$ is stable then $\wedge^2 W$ is polystable. Hence the family $S_{\cW}$ defines a morphism $\tM \to \MO^0 ( 5, \Oc )$ which factorizes via a morphism $\Theta_5^0 \colon \MS ( 4, \Oc )_\st \to \MO^0 ( 5, \Oc )$.

Next, note that if $W$ and $W'$ are S-equivalent semistable $\Oc$-valued symplectic bundles, then
\[ \gr (S_W) \oplus \Oc \ = \ \gr ( \wedge^2 W ) \ \cong \ \gr ( \wedge^2 W' ) \ = \ \gr ( S_{W'} ) \oplus \Oc . \]
Therefore, if $W \sim W'$ then $S_W \sim S_{W'}$. Now by \cite{BLS}, the moduli space $\MS ( 4, \Oc )$ has Picard number one. Thus part (a) can be proven similarly as before.
%
The proof of (b) is similar. 

As before, the target component is determined by the parity of degree of rank two isotropic subbundles. The details are described in Theorem \ref{rk5Structure}. \end{proof}

As in previous cases, we now describe how to construct an ``inverse'' for the operation $W \mapsto S_W$, which is valid for an arbitrary $L$-valued orthogonal bundle of rank five.

\begin{theorem} \label{rk5Structure} Let $L$ be a line bundle of degree $2m$, and $V$ a rank five $L$-valued orthogonal bundle of determinant $M^5$, where $M^2 \cong L$.
\begin{enumerate}
\item[(a)] The bundle $V$ admits an isotropic subbundle $E$ of rank two.
\item[(b)] If $\deg(E) \equiv m \mod 2$, then there is an $M$-valued symplectic bundle $(W, \alpha)$ of rank four (and determinant $L$) such that $V \cong S_W$.
\item[(c)] If $\deg(E) \not\equiv m \mod 2$, then there is an $M(-x)$-valued symplectic bundle $(W_1, \alpha_1)$ of rank four (and determinant $L(-2x)$) such that $V \cong S_{W_1} \otimes \Oc(x)$.
\end{enumerate} \end{theorem}

\begin{proof} As $M^2 \cong L$, in fact $M$ itself is an $L$-valued orthogonal bundle. The orthogonal direct sum $V \perp M$ is a rank six orthogonal bundle of determinant $L^3$. Then $V \perp M$ admits a rank three isotropic subbundle $\bE$ by Proposition \ref{DetIrns}. It is easy to see that
\[ E \ := \ \bE \cap V \ = \ \Ker \left( \bE \ \to \ ( V \perp M ) \ \to \ M \right) \]
is an isotropic subbundle of rank two in $V$. Thus we obtain (a).

Note moreover that $\deg (E) + m = \deg ( \bE )$. 
 Suppose $\deg(E) \equiv m \mod 2$. Then $\deg ( \bE )$ is even. By Theorem \ref{rk6Structure} (a), there exists a bundle $W$ of rank four and determinant $L$ such that $V \perp M \cong \wedge^2 W$. 

Now the subbundle $M \subset \wedge^2 W$ is nowhere isotropic with respect to the orthogonal structure, being an orthogonal direct summand. By Proposition \ref{rk6Envelop}, vectors in $\wedge^2 W$ isotropic with respect to the orthogonal structure correspond precisely to the decomposable tensors. Therefore, the image of $M \to \wedge^2 W$ is indecomposable at all points. Tensoring by $W^*$ and contracting, we obtain an antisymmetric isomorphism
\[ M \otimes W^* \ \to \ \left( \wedge^2 W \right) \otimes W^* \ \to \ W , \]
equivalently, an $M$-valued symplectic structure on $W$. By Proposition \ref{wedge2W}, we have
\[ V \perp M \ = \ \wedge^2 W \ = \ S_W \perp M . \]
By uniqueness of direct sum decompositions, we have $V \cong S_W$ as vector bundles; and by \cite[Proposition 3.1]{Gro} they are isomorphic as orthogonal bundles. This proves (b).

If $\deg(E) \not\equiv m \mod 2$, then $\deg ( \bE )$ is odd. Then $V(-x) \perp M(-x)$ is $L(-2x)$-valued orthogonal and admits the isotropic subbundle $\bE(-x)$ which is of rank three and even degree. By part (a), there exists a rank four $M(-x)$-valued symplectic bundle $W_1$ such that $V(-x) \cong S_{W_1}$, as desired. \end{proof}

\subsection{Enveloping bundle structure}

Let $W$ be a rank four $M$-valued symplectic bundle as above, and $V = S_W$ the associated rank five $M^2$-valued orthogonal bundle. We now determine the quadric fibration $Q_V \subset \PP V$.

By Proposition \ref{rk6Envelop}, the projective bundle $\PP ( \wedge^2 W )$ is an enveloping bundle for $Q_{\wedge^2 W} = \Gr (2, W)$. As the orthogonal structure on $V$ is inherited from $\wedge^2 W$, we have
\[ Q_{S_W} \ = \ Q_{\wedge^2 W} \cap \PP V . \]

Now recall from \cite[p.\ 759 ff.]{GH} that a \textsl{linear line complex in $\PP^3$} is a three-dimensional family of lines defined by a hyperplane section of the Pl\"ucker image of $\Gr ( 2, \C^4 )$ in $\PP^5$. The datum of a linear line complex is equivalent to a choice of bilinear antisymmetric form on $\C^4$ up to scalar, which is nondegenerate if and only if the intersection is smooth; equivalently, the hyperplane is not tangent to $\Gr ( 2 , \C^4 )$ at any point. By the above discussion, for each $p \in C$, the hyperplane $\PP V \subset \PP ( \wedge^2 W )$ determines a smooth linear line complex in $\PP W|_p$. Clearly this is isomorphic to the Lagrangian Grassmannian $\LG (W|_p)$ parameterizing two-dimensional $\alpha$-isotropic subspaces of $W|_p$. Summarizing, we have:

\begin{proposition} \label{rk5Envelop} Let $V = S_W$ be as above. Then the projectivization of the locus of isotropic vectors in $\PP V$ is the Lagrangian Grassmann bundle $\LG ( W )$. \end{proposition}

\subsection{Isotropic subbundles of \texorpdfstring{$S_W$}{S\_W}}

Let $W$ be a rank four $M$-valued symplectic bundle. To classify rank two isotropic subbundles of the rank five $M^2$-valued orthogonal bundle $S_W$, we could use a strategy similar to that in \cite[{\S} 8]{CCH3}. However, we offer instead an approach exploiting the geometry of the enveloping bundle, which also gives further insight into the rank six case.

Let $E \subset S_W$ be a rank two isotropic subbundle. Then $\PP E$ is a $\PP^1$-bundle contained in $\LG ( W )$.  As above, by nondegeneracy of the symplectic form, for each $p \in C$ the intersection $\PP ( S_W|_p ) \cap \Gr (2, W|_p)$ is smooth. Then by \cite[p.\ 759]{GH}, there exist a line $N_p$ and a hyperplane $H_p$ in $W|_p$ such that $\PP W$ parameterizes the pencil of two-dimensional (isotropic) subspaces of $V|_p$ forming a complete flag
\[ 0 \ \subset \ \ N_p \ \subset \ \Lambda \ \subset \ H_p \ \subset \ W|_p . \]
It follows that
\begin{equation} E|_p \ = \ \{ v \wedge w \in \wedge^2 W : v \in N_p , w \in H_p \} \ \cong \ N_p \otimes \frac{H_p}{N_p} . \label{RankFiveIsotDescr} \end{equation}

We claim moreover that $H_p = N_p^\perp$, the orthogonal complement of $N_p$ with respect to the symplectic form $\alpha$. To see this: $H_p$ contains a pencil of two-dimensional subspaces containing $N_p$ which are moreover isotropic, so in particular annihilate $N_p$. As $H_p$ has dimension three, it is spanned by elements of $N_p^\perp$. As $\dim ( N_p^\perp ) = \dim H_p$, we have equality. Using local triviality, we see that the union of the subspaces $N_p$ is a (uniquely determined) line subbundle $N \subset W$.

We now define a map $\lambda_N \colon N \otimes N^\perp \to \wedge^2 W$ by
\[ \lambda_N ( v \otimes w ) \ = \ v \wedge w . \]
Clearly $\lambda_N$ is the restriction to $N \otimes N^\perp$ of the map $\tau_N$ defined in (\ref{TauN}), and
\[ \lambda_N ( N \otimes N^\perp ) \ = \ \tau_N ( N \otimes V ) \cap \left( \wedge^2 N^\perp \right) \ \cong \ N \otimes \frac{N^\perp}{N} . \]
Now as $N^\perp / N$ is $M$-valued symplectic of rank two, it has determinant $M$ and degree $m$. Thus $\deg \Image ( \lambda_N ) = 2 \cdot \deg N + m$.

Conversely, given any line subbundle $N \subset W$, the subbundle $\lambda_N ( N \otimes N^\perp )$ is a rank two subbundle of $\wedge^2 W$. As $\alpha ( v \wedge w ) = 0$ for any $v \in N$ and $w \in N^\perp$, we have $\Image \left( \lambda_N \right) \subseteq \Ker \left( \alpha \colon \wedge^2 W \to M \right) = S_W$.  As moreover every element of $\Image ( \lambda_N )$ is decomposable, $\PP \Image ( \Lambda_N )$ is contained in $Q_{\wedge^2 W} \cap \PP S_W = \LG ( W )$. Hence $\Image ( \Lambda_N )$ is isotropic by Lemma \ref{SimpleIsotropyCrit}. Thus we may summarize as follows.

\begin{theorem} \label{rk5IsotSubbs} The association $N \mapsto \lambda_N ( N \otimes N^\perp )$ defines isomorphisms $\Quoto_{1,d} ( W ) \isom \IQo_{2, 2d + m} ( S_W )$. Note also that  the association 
$N \mapsto \Image \left( \lambda_N \right) \otimes \Oc (x)$ defines an isomorphism
\[ \Quoto_{1,d} ( W ) \cong \IQo_{2, 2d+m+2} ( S_W \otimes \Oc (x) ) . \] \end{theorem}

\begin{remark} \label{LinkIsotSubbsFiveSix} Let $E$ be a rank two isotropic subbundle of $S_W$. We observe from the description (\ref{RankFiveIsotDescr}) that
\[ E \ = \ \tau_N ( N \otimes V ) \cap \left( \wedge^2 N^\perp \right) ; \]
that is, $E$ is the intersection of two isotropic rank three subbundles $\wedge^2 N^\perp$ and $N \otimes \frac{W}{N}$ of the rank six orthogonal bundle $\wedge^2 W$, as described in Theorem \ref{rk6IsotSubbs}. Moreover, in view of \cite[Proposition, p.\ 735]{GH}, the fact that
\[ \dim \left( \wedge^2 \left( N^\perp|_p \right) \cap \lambda_N ( N \otimes N^\perp ) \right) \ = \ 2 \ \not\equiv 3 \mod 2 \]
reflects the fact that these two subbundles define elements of opposite components of the orthogonal Grassmannian $\OG (3, \wedge^2 V|_p )$ at each point. \end{remark}

Next, by Theorem \ref{rk5Envelop}, an isotropic line subbundle of $S_W$ is nothing but a section of the Lagrangian Grassmannian bundle $\LG (W) \to C$. This equivalent to a choice of rank two isotropic subbundle $E \subset W$, and then $\wedge^2 E$ is the isotropic line subbundle of $S_W$. Thus we have:

\begin{theorem} \label{rk5rk1} Let $W$ be an $M$-valued symplectic bundle and $S_W$ the associated rank five $M^2$-valued orthogonal bundle. Then for each $d$, the assignment $E \mapsto \wedge^2 E$ defines an isomorphism
\[ \LQ_d^\circ (W) \ \isom \ \IQo_{1,d} ( S_W ) , \]
where $\LQ_d^\circ (W)$ is the subscheme of $\Quoto_{2,d}(W)$ consisting of Lagrangian (maximal rank isotropic) subbundles of the symplectic bundle $W$. \end{theorem}

\subsubsection{Enumeration of maximal isotropic subbundles} By Theorem \ref{rk5IsotSubbs}, the number of rank two isotropic subbundles of maximal degree of $V$ equals the number of maximal line subbundles of the rank four symplectic bundle $W$. However, the counting formula in Theorem \ref{MaxLineSubbs} is valid only for a general vector bundle of rank
four, while the symplectic bundles of rank four are contained in a proper closed subset of the moduli of vector bundles, so we cannot apply Theorem \ref{MaxLineSubbs} directly. Therefore, we leave the following as a conjecture.

\begin{conjecture} \label{rk5enum} Let $V$ be a stable rank five $\Oc$-valued orthogonal bundle which is general in its component of moduli.
\begin{enumerate}
\item[(a)] Suppose $g \equiv 1 \mod 4$. Suppose $w_2 (V) = 0$, so that $V \cong S_W$ for a $W \in \MS ( 4, \Oc )$ which may be assumed to be general. Then $V$ has $4^g$ maximal rank two isotropic subbundles, each of degree $- \frac{3}{2}(g-1)$.
\item[(b)] Suppose $g \equiv 3 \mod 4$. Suppose $w_2 (V) = 1$, so that $V \cong S_{W_1} \otimes \Oc (x)$ for some $W_1 \in \MS ( 4, \Oc (-x))$ which may be assumed to be general. Then $V$ has $4^g$ rank two isotropic subbundles of maximal degree $-\frac{3}{2}(g-1)$.
\end{enumerate}
\end{conjecture}

\noindent On the other hand, for isotropic line subbundles we have the following.

\begin{theorem} Let $V$ be a rank five {\color{blue} $\cO_C$-valued} orthogonal bundle which is general in its component of moduli. Then each isotropic line subbundle has degree $1-g$, and the number of such subbundles is given by
\[ \begin{array}{|c|c|c|} & g \hbox{ even} & g \hbox{ odd} \\
\hline
 w_2 (V) = 0 & 2^{g-1} \cdot ( 3^g - 1 ) & 2^{g-1} \cdot ( 3^g + 1 ) \\
\hline
 w_2 (V) = 1 & 2^{g-1} \cdot ( 3^g + 1 ) & 2^{g-1} \cdot ( 3^g - 1 )
\end{array} \]
\end{theorem} 

\begin{proof} Suppose $w_2 (V) = 1$. By Theorem \ref{rk5moduli} (b), we may assume that $V = S_W \otimes \Oc ( x )$ for a rank four $\Oc(-x)$-valued symplectic $W$ which is general in moduli. By Theorem \ref{rk5rk1}, each maximal isotropic line subbundle is of the form $(\wedge^2 E) \otimes \Oc (x)$ for a maximal Lagrangian subbundle $E \subset W$. By \cite[Proposition 3.2]{CCH2} we have $\deg (E) = -g$, whence $\deg ( \wedge^2 E \otimes \Oc (x) ) = 1-g$; and the number of such $E$ is given by \cite[Corollary 6.3 {\color{blue}(2) and (4)}]{CCH2}. The case where $w_1 ( V ) = 0$ is similar. \end{proof}

\section{Higher rank}

We shall now conclude by formulating some conjectures on orthogonal bundles of higher rank and their isotropic Quot schemes in general, based on the extra geometric information for low rank gathered in the previous sections. We shall use results from the literature to a greater extent than before.

Let $i \in \{ 0 , 1 \}$. By \cite[{\S} 5.4]{CH14} and \cite[Theorem 6.1]{CH15}, a general $V \in \MO^i ( 2n , \Oc )$ or $\MO^i ( 2n-1 , \Oc )$ has a finite number of isotropic subbundles of rank $n$ and maximal degree $-\frac{1}{2}n(g-1)$ under the following congruence condition:
\[
n (g-1) \equiv 2i \mod 4
\]     
We denote this number by  $N^0(g,r)$, where $r = 2n$ or $2n-1$.


Theorems \ref{rk4enum}, \ref{rk3enum} and \ref{rk6enum} and Conjecture \ref{rk5enum} suggest the following equalities for arbitrary rank:

\begin{conjecture} \label{ConjOne} Let $g, r, n, i$ be as above. For all $r \ge 3$, we have
\[
N^0(g, 2n) \ = \ 2 \cdot N^0(g, 2n-1) . 
\]
\end{conjecture}

An intuition behind the  conjecture is as follows. We have briefly mentioned the orthogonal Grassmannian bundle $\OG (V)$ of an orthogonal bundle $V$, which parameterizes isotropic subspaces of maximal dimension in fibers of $V$; equivalently, projective linear subspaces of maximal dimension contained in the quadrics $Q_V|_x$. If $V$ has rank $2n$ then by \cite[Proposition 2.12]{CCH3}, the fibration $\OG (V)$ has two connected components $\OG(V)_0$ and $\OG(V)_1$. (This has a particularly clear geometric interpretation in rank four; see Remark \ref{rk4OG}.) Following \cite[Definition 3.1]{CCH3}, we define $\IQoe_\delta$ to be the component of $\IQoe$ parameterizing subbundles belonging to $\OG(V)_\delta$. By \cite[Proposition 3.4 (3)]{CH15} and \cite[]{CCH3}, for all $e$, there are canonical isomorphisms
\[ \IQo_{n,e} \left( V \perp \Oc \right)_\delta \ \isom \ \IQo_{n-1, e} ( V ) \]
for $\delta \in \{ 0, 1 \}$. (See Remark \ref{LinkIsotSubbsFiveSix} for related discussion for rank five.) In particular, the number of maximal isotropic subbundles of $V \perp \Oc$ is exactly twice that of $V$. However, since a  bundle of the form $V \perp \Oc$ is not general in $\MO^i ( 2n , \Oc )$, we need an argument on the conservation of number.

The same argument suggests:

\begin{conjecture} \label{Conjtwo} If a general $V \in \MO ( 2n, \Oc)$ has a finite number $N^0(g, 2n)$ of rank $n$ maximal isotropic subbundles, then exactly half of them define sections of $\OG ( V )_0$ and another half those of $\OG ( V )_1$. \end{conjecture}


In a forthcoming paper with D.\ Cheong, we propose to compute the numbers $N^i ( r, g )$ more generally and to investigate the validity of the above conjectures in general.


\begin{thebibliography}{99}

\bibitem[BS02]{BS} V.\ Balaji and C.\ S.\ Seshadri: \textsl{Semistable principal bundles - I (characteristic zero)}. J.\ of Algebra \textbf{258} (2002), 321--47.

\bibitem[Bea08]{Bea} A.\ Beauville: \textsl{Orthogonal bundles on curves and theta functions}. Ann.\ Inst.\ Fourier \textbf{56}, no.\ 5 (2008), 1405--1418.

\bibitem[BLS89]{BLS} A.\ Beauville, Y.\ Laszlo and C.\ Sorger: \textsl{The Picard group of the moduli of $G$-bundles on a curve}. Compos.\ Math.\ \textsl{112}, no.\ 2 (1998), 183--216.

\bibitem[BG10]{BG} I.\ Biswas and T.\ L.\ Gomez: \textsl{Hecke transformation for orthogonal bundles and stability of Picard bundles}. Commun.\ Analytic Geom.\ \textbf{18}, no.\ 5 (2010), 857--890.

\bibitem[CH14]{CH14} I.\ Choe and G.\ H.\ Hitching: \textsl{A stratification on the moduli spaces of symplectic and  orthogonal bundles over a curve}. International J.\ of Math., \textbf{25}, no.\ 5 (2014), 27 pp.

\bibitem[CH15]{CH15} I.\ Choe and G.\ H.\ Hitching: \textsl{Maximal isotropic subbundles of orthogonal bundles of odd rank over a curve}. International J.\ of Math., \textbf{26}, no.\ 13  (2015), 23 pp.

\bibitem[CCH21a]{CCH2} D.\ Cheong, I.\ Choe and G.\ H.\ Hitching: \textsl{Counting Lagrangian subbundles over an algebraic curve}. J.\ of Geometry and Physics (2021); DOI: 10.1016/j.geomphys.2021.104288

\bibitem[CCH21b]{CCH3} D.\ Cheong, I.\ Choe and G.\ H.\ Hitching: \textsl{Isotropic Quot schemes of orthogonal bundles over a curve}. International J.\ of Math.\ (2021); DOI: 10.1142/S0129167X21500476.

\bibitem[DN89] {DN} J.-M.\ Drezet and M.\ S.\ Narasimhan: \textsl{Groupe de Picard des vari\'et\'es de modules de fibr\'es semi-stables sur les courbes alg\'ebriques}. Invent.\ Math.\ \textbf{97}, no.\ 1 (1989), 53--94.

\bibitem[FH91]{FH} W.\ Fulton and J.\ Harris: \textsl{Representation Theory -- A First Course}. GTM 129, Springer-Verlag, New York, 1991.


\bibitem[GHS03]{GHS} T.\ Graber, J.\ Harris, and J.\ Starr: \textsl{Families of rationally connected varieties}. J.\ Amer.\ Math.\ Soc.\ \textbf{16} (2003), no.\ 1, 57-67.

\bibitem[Gro57]{Gro} A.\ Grothendieck: \textsl{Sur la classification des fibr\'es holomorphes sur la sph\`ere de Riemann}. Amer.\ J.\ Math.\ \textbf{79} (1957), 121--138.

\bibitem[Gro58]{Gro58} A.\ Grothendieck: \textsl{A General Theory of Fibre Spaces with Structure Sheaf}, second ed. NSF report, Univ.\ of Kansas, 1958. Available at \texttt{https://webusers.imj-prg.fr/\~{}leila.schneps/grothendieckcircle/Kansasnotes.pdf} .

\bibitem [GH78]{GH} P.\ Griffiths; J.\ Harris: \textsl{Principles of Algebraic Geometry}. Wiley, USA, 1994.

\bibitem [Har77]{Har} R.\ Hartshorne: \textsl{Algebraic geometry} Graduate Texts in Mathematics 52. New York-Heidelberg-Berlin: Springer-Verlag (1983).

\bibitem[Hit07]{Hit} G.\ H.\ Hitching: \textsl{Subbundles of symplectic and orthogonal vector bundles over curves}. Math. \ Nachr.  \ \textbf{280}, no. \ 13-14 (2007), 1510--1517.

\bibitem[Hir96]{Hir} Hirschowitz, A.: Probl\`emes de Brill--Noether en rang sup\'erieur. Pr\'epublications Math\'ematiques n.\ 91, Nice (1986).

\bibitem[Hol04]{Hol} Y.\ I.\ Holla: \textsl{Counting maximal subbundles via Gromov--Witten invariants}. Math.\ Ann.\ \textbf{328}, no.\ 1-2 (2004), 121--133.

\bibitem[LN03]{LNew} H.\ Lange and P.\ E.\ Newstead: \textsl{Maximal subbundles and Gromov--Witten invariants}. V.\ Lakshmibai (ed.) et al., A tribute to C. S. Seshadri. A collection of articles on geometry and representation theory. Basel: Birkh\"auser. Trends in Mathematics (2003), 310--322.

\bibitem[LM20]{LM} A.\ Lanteri and R.\ Mallavibarrena: \textsl{Projective bundles enveloping rational conic fibrations and osculation}. J.\ Pure Appl.\ Algebra \textbf{224}, no.\ 12 (2020), 20 pp.

\bibitem[LMP15]{LMP} A.\ Lanteri, R.\ Mallavibarrena, and R.\ Piene: \textsl{Inflectional loci of quadric fibrations}. J.\ Algebra \textbf{441} (2015), 363--397.


\bibitem[Mum71]{Mum} D.\ Mumford: \textsl{Theta characteristics of an algebraic curve}. Ann.\ Sci.\ \'Ec.\ Norm.\ Sup\'er.\ (4) \textbf{4} (1971), 181--192.

\bibitem[Oxb00]{Ox} W.\ M.\ Oxbury: \textsl{Varieties of maximal line subbundles}. Math. \ Proc. \ of the Cambridge Phil. \ Soc. \ \textbf{129} (2000), 9--18. 

\bibitem[Ram81]{Ram} S.\ Ramanan: \textsl{Orthogonal and spin bundles over hyperelliptic curves}. Proc.\ Indian Acad.\ Sci., Math.\ Sci.\ \textbf{90} (1981), 151--166.

\bibitem[Rth96]{Rth} A. Ramanathan: \textsl{Moduli for principal bundles over algebraic curves: I \& II}. Proc.\ Indian Acad.\ Sci., Math.\ Sci.\ \textbf{106} (1996), no.\ 3, 301--328 and no.\ 4, 421--449.



\bibitem[Serm12]{Serman} O.\ Serman: \textsl{Orthogonal and symplectic bundles on curves and quiver representations}. Brion, M.\ (ed.), ``Geometric methods in representation theory'' II. Selected papers based on the presentations at the summer school, Grenoble, France, June 16 -– July 4, 2008. Paris: Soc.\ Math.\ de France. S\'eminaires et Congr\`es \textbf{24}, pt.\ 2  (2012), 393--418.

\bibitem[Serre90]{Ser} J.-P. Serre, \textsl{Rev\^{e}tements \`{a} ramification impaire et th\^{e}ta-caract\'{e}ristiques}. C. R. Acad. Sci. Paris S\'{e}r. I Math.   \textbf{311}, no. \ 9 (1990), 547--552.


\end{thebibliography}
\end{document}